\documentclass[10pt]{article}
\usepackage{amsmath}
\usepackage{amsthm}
\usepackage{amssymb}

\newtheorem{problem}{Problem}
\newtheorem{theorem}{Theorem}[section]
\newtheorem{proposition}[theorem]{Proposition}
\newtheorem{definition}[theorem]{Definition}
\newtheorem{corollary}[theorem]{Corollary}
\newtheorem{lemma}[theorem]{Lemma}
\newtheorem{remark}[theorem]{Remark}

\newcommand{\mL}{{\cal L}}

\newcommand{\R}{\mathbb{R}}

\newcommand{\El}{{\cal L}}

\newcommand{\beq}{\begin{equation}}
\newcommand{\eeq}{\end{equation}}

\newcommand{\C}{\mathbb{C}}
\newcommand{\N}{\mathbb{N}}

\newcommand{\mD}{{\mathcal{D}}}
\newcommand{\mF}{{\cal F}}
\newcommand{\mM}{{\cal M}}
\newcommand{\mW}{{\cal W}}

\title {\bf Identification of a convolution kernel in a control problem for the heat equation with a boundary memory term
\thanks{This work partially was supported by the
MIUR-PRIN Grant 20089PWTPS
{\it Analisi Matematica nei Problemi Inversi per le Applicazioni} }}

\author{Cecilia Cavaterra\\
Dipartimento di Matematica, Universit\`{a} degli Studi di Milano\\
Via Saldini 50, 20133 Milano, Italy
\\
cecilia.cavaterra@unimi.it
\and
Davide Guidetti\\
Dipartimento di Matematica, Universit\`{a} di Bologna\\
Piazza di Porta San Donato 5, 40126 Bologna, Italy
\\
davide.guidetti@unibo.it}
\date{}

\begin{document}

\maketitle

\begin{abstract}
We consider the evolution of the temperature $u$ in a material with thermal memory
characterized by a time-dependent convolution kernel $h$.
The material occupies a bounded region $\Omega$ with a feedback device controlling
the external temperature located on the boundary $\Gamma$.
Assuming both  $u$ and $h$ unknown, we formulate an inverse control problem for an
integrodifferential equation with a nonlinear and nonlocal boundary condition.
Existence and uniqueness results of a solution to the inverse problem are proved.
\end{abstract}



\section{Introduction}
In this paper we want to study the evolution of the temperature $u$ in a material with thermal memory,
occupying a bounded region $\Omega \subseteq \R^3$.
The memory mechanism is characterized by a time-dependent convolution kernel $h$.

Here we are interested in analyzing the heat exchange at the boundary $\Gamma:= \partial \Omega$ under the influence
of a thermostat, with boundary conditions of the third type. On account of the existing literature,
(see, e.g., \cite{CoGrSp1}, \cite{HoNiSp1} and their references), we are willing to formulate and study an automatic
control problem based on a feedback device located on the boundary $\Gamma$.
This device is prescribed by means of a quite general memory operator (in the following we shall give precise definitions).
Moreover, due to the presence of a memory term in the evolution equation for $u$, the past history involves also the boundary
condition that turns out to be of integrodifferential type.

Therefore, we introduce the following initial and boundary value problem:
\begin{problem}\label{P1}
Find $u$ of domain $Q_T$ such that
\begin{equation}
\left\{ \begin{array}{ll}
D_t u(t,x) = Au(t,x) + (h \ast Au)(t,x) + f(t,x), & (t,x) \in Q_T, \\ \\
Bu(t,x) + (h \ast Bu)(t,x) + q(t,x) = u_e(t,x) - u(t,x), & (t,x) \in \Sigma_T, \\ \\
u(0,x) = u_0(x), & x \in \Omega.
\end{array}
\right.
\end{equation}
where
\begin{equation}\label{eq1.2A}
Q_T:= (0, T) \times \Omega,
\end{equation}
\begin{equation}\label{eq1.3A}
\Sigma_T:= (0, T) \times \Gamma,
\end{equation}
and
\begin{equation}
(h \ast f)(t,x) := \int_0^t h(t-s) f(s,x) ds.
\end{equation}
\end{problem}
\noindent
Here $T >0$, $f : Q_T \to \mathbb{R}$ is the heat source, $u_0 : \Omega \to \mathbb{R}$ is the initial temperature,
$q$ accounts for the past history of $u$ on the boundary up to $t = 0$, while $u_e$ represents the temperature
of the external environment. Moreover, $A$ and $B$ are linear differential operators defined by
\begin{equation}
A = \sum_{i=1,j}^n D_{x_i} (a_{ij}(x) D_{x_j}), \quad x \in \Omega,
\end{equation}
\begin{equation}
B = \sum_{i=1}^n b_{i}(x) D_{x_i} + b_0(x), \quad x \in \Gamma.
\end{equation}
As already stated above,  the external control $u_e$ should be regulated by a feedback device based on measurements of $u$.
Suppose that we are able to measure the temperature by a real system of thermal sensors placed in some fixed positions over
$\Omega$ and/or $\Gamma$ as follows:
\begin{equation}\label{eqA1.7}
\mM(u)(t):= \int_\Omega \omega_1(x) u(t,x) dx + \int_{\Gamma} \omega_2 (y) u(t,y) d\sigma,
\end{equation}
where $\omega_1$ and $\omega_2$ are functions with domain $\Omega$ and $\Gamma$, respectively.
We consider a thermostat device modifying $u_e$, on account of $\mM(u)$, in this way  (see, e. g. \cite{HoNiSp1},
\cite{CoGrSp1} and their references):
\begin{equation}
u_e = \phi u_A + u_B \quad \mbox{\rm on } \Sigma_T.
\end{equation}
Here $u_B: \Sigma_T \to \mathbb{R}$ is a given reference boundary value (e.g. the external average temperature),
while $u_A : \Sigma_T \to \mathbb{R}$ is a (known) factor of the part of $u_e$ that can be controlled by our device.
The dynamic control is exerted through the function $\phi: [0, T]Ê\to \mathbb{R}$, solution to the problem
\begin{equation}\label{eq0.9}
\left\{\begin{array}{ll}
\epsilon \phi' + \phi = \mW (\mM (u))Ê+ u_C & \mbox{ in } [0, T], \\ \\
\phi(0) = \phi_0,
\end{array}
\right.
\end{equation}
where $u_C : [0, T] \to \mathbb{R}$ is a given function, $\epsilon$ is a positive parameter and $\phi_0 \in \mathbb{R}$.

The nonlinear operator $\mW$ completes the description of the feedback action. We assume that $\mW$ is a memory operator,
accordingly to the definition in \cite{Vi1}, Chap. III, to be precisely defined in the following.
Mathematical literature contains several examples of operators $\mW$ useful in applications:
we mention the generalized plays and Preisach operators (see \cite{Vi1}, Chaps. III-IV
and \cite{KrPo1}, Part 1).

Going back to the Cauchy problem for $\phi$, we formally deduce that $u_e$ is assigned by
\begin{equation}\label{eq0.10}
u_e = \mF(\mW (\mM(u))) \quad \mbox{\rm on } \Sigma_T,
\end{equation}
where
\begin{equation}\label{eqA1.11}
\mF(r)(t,y) :=    (E_1 \ast r)(t) u_A(t,y) + E_0(t,y), \quad (t,y) \in \Sigma_T,
\end{equation}
and
\begin{equation}\label{eqA1.12}
E_1(t):= \epsilon^{-1}Êe^{-t/\epsilon},
\end{equation}
\begin{equation}\label{eqA1.13}
E_0(t,y) = [ (E_1 \ast u_C)(t) + \epsilon \phi_0 E_1(t)]Êu_A(t,y) + u_B(t,y).
\end{equation}
Therefore, on account of (\ref{eq0.10}), the feedback nonlinear control problem reduces to a system with a nonlinear,
nonlocal boundary condition. In the paper \cite{CaCo1}, the authors studied a problem in the form (\ref{P1}), with
two possible choices of the memory operator: the relay switch operator or the Preisach operator.
Moreover, in the second part of the paper they studied the case in which they had also to identify
a time-dependent factor of the heat source.

Instead, here we assume that the memory kernel $h$ is unknown. In order to determine $h$ along with the temperature $u$,
we need an additional information: we assume to know the following quantity
\begin{equation}\label{eqA1.14}
\Phi(u(t)) := \int_\Omega \omega(x) u(t,x) dx, \quad  \text{for any } t \in [0, T],
\end{equation}
where $\omega$ is a properly smooth function, vanishing together with its first derivatives in $\Gamma$.

We can now formulate our inverse control problem:

\begin{problem}\label{P2}
Find $u$ of domain $Q_T$ and $h$ of domain $[0, T]$ such that
\begin{equation}\label{eqA1.15}
\left\{ \begin{array}{ll}
D_t u = Au + h \ast Au  + f, & \mbox{ in } Q_T, \\ \\
Bu  + h \ast Bu  + q  = \mF (\mW (\mM (u))) - u, &  \mbox{ on }  \Sigma_T, \\ \\
u(0,\cdot) = u_0, & \mbox{ in }  \Omega\\ \\
\Phi(u) = g & \mbox{ in } [0, T].
\end{array}
\right.
\end{equation}
\end{problem}

In order to solve our problem we need to apply the theory developed by Lions and Magenes in
\cite{LiMa1} (see also \cite{LiMa2}). To this aim, we have to settle all the functions,
the operators and the linear spaces involved in Problem 2 in the framework of a complex
context.
It is not difficult to realize that whence all the functions and coefficient appearing in
Problem 2 are real valued, then the real part of the solution turns out to be real valued as well.

Hence, from now on we will assume all the functions introduced before taking values in $\mathbb{C}$.

As far as the operator $\mW$ is concerned, even if in applications it is defined only for functions
$f: [0,T] \to \mathbb{R}$, nevertheless, whenever $f: [0,T] \to \mathbb{C}$, then we intend
$\mW (f)$ as $\mW(Re f) + i \mW(Im f )$.

More in details, a memory operator $\mW_\tau$ is characterized as follows.

We indicate by $C([0, \tau])$, $\tau \in \R^+$, the Banach space of continuous complex valued functions
of domain $[0, \tau]$, equipped with its standard norm. Then

\medskip

{\it (C1) $\forall \, \tau \in [0, T]$,   $\mW_\tau: C([0, \tau]) \cap \cap BV([0, \tau]) \to C([0, \tau]) \cap BV([0, \tau])$ is given, where $BV([0, \tau])$ stands for the class
of bounded variation functions.;

\medskip

(C2) if \, $0 \leq \tau_2 \leq \tau_1 \leq T$ and $f \in C([0, \tau_1]) \cap BV([0, \tau_1])$, then $W_{\tau_2}(f_{|[0, \tau_2]}) = [W_{\tau_1}(f)]_{|[0, \tau_2]}$;}

\medskip

{\it (C3) There exists $L \in \R^+$, such that $\forall \tau \in [0, T]$, $\forall f,g \in C([0, \tau]) \cap BV([0, \tau])$,
$$
\quad \|\mW_\tau(f) - \mW_\tau(g)\|_{C([0, \tau])} \leq L\|f - g\|_{C([0, \tau])}.
$$}

\medskip

\begin{remark}
{\rm Conditions which are alternative to (C1)-(C3) can be adopted, in order to obtain the conclusion of the main result of the paper,  Theorem  \ref{thA2.1}. A short discussion in this direction will be put
in the final Remark \ref{ref}.

}

\end{remark}

\medskip

On account of ${\it (C2)}$, given $f \in C([0, \tau_1]) \cap BV([0, \tau_1])$, then  $[\mW_{\tau_1}(f)]_{|[0, \tau_2]}$ depends only on $f_{|[0, \tau_2]}$.
Hence, if $f \in C([0, \tau]) \cap BV([0, \tau])$ for some $\tau \in [0, T]$, we shall loosely write $\mW(f)$ in alternative to $\mW_\tau(f)$.

We conclude this introduction by fixing the basic notations, recalling some well known definitions and facts, and outlining the organization of the paper.

Concerning the notation, we indicate with $\N$ and $\N_0$ the sets of positive and nonnegative integers, respectively. If $\beta \in \R$,
$[\beta]$ stand for its integer part and $\{\beta\} := \beta -[\beta]$.

We indicate with $C$ a positive constant which may be different from time to time. However, in a sequence of estimates, we write also $C_1, C_2, \dots$.
In order to stress the fact that the constant $C$ depends on $\alpha, \beta, \dots$, we shall write $C(\alpha,\beta,\dots)$.

We  indicate with $BV([0, T])$ the class of complex valued bounded variation functions with domain $[0, T]$.
If $E$ is a Banach space, $\alpha \in (0, 1)$ and $f : [0, T]Ê\to X$, we set
$$
[f]_{C^\alpha([0, T]; E)} := \sup_{0 \leq s < t \leq T} \frac{\|f(t) - f(s)\|_E}{(t-s)^\alpha},
$$ and, in case $[f]_{C^\alpha([0, T]; E)} < \infty$, we write $f \in C^\alpha([0, T]; E)$. If $E = \C$, we simply write $C^\alpha([0, T])$.

If $E$ and $F$ are normed spaces, we indicate with $\El(E,F)$ the space of linear bounded operators from $E$ to $F$.
If $E = F$, we simply write $\El(E)$. We  indicate with $E'$ the space of continuous antilinear functionals in $E$, equipped with its natural norm.

Let $\Omega$ be an open subset of $\R^n$. We consider the Sobolev spaces $H^m(\Omega)$ ($m \in \N_0$), defined as
$$H^m(\Omega) = \{u \in L^2(\Omega): D^\alpha u \in L^2(\Omega), |\alpha| \leq m\},$$
with $D^\alpha$ intended in the sense of distributions.
$H^m(\Omega)$ is a Hilbert space with the norm
$$
\|u\|^2_{H^m(\Omega)} := \sum_{|\alpha| \leq m} \|D^\alpha u\|_{L^2(\Omega)}^2.
$$
Let $\beta \in \R^+$. Then we define
$$
H^\beta(\Omega):= (H^{[\beta]}(\Omega), H^{[\beta]+1}(\Omega))_{\{\beta\},2},
$$
denoting with $(\cdot, \cdot)_{\theta,2}$ ($0 < \theta < 1$) the real interpolation functor.
This definition is equivalent to the one in \cite{LiMa1}, Chap. 1.9 (see \cite{Gr1}, 1.2) .
In the case $\Omega = \R^n$, $H^\beta(\Omega)$ admits a well known characterization in terms of
Fourier transform (see \cite{LiMa1}, Chap. 1.7).
If $\alpha \in \N_0^n$ and $|\alpha| \leq \beta$, $D^\alpha \in \mL(H^\beta(\Omega), H^{\beta-|\alpha|}(\Omega))$.

Given an open subset $\Omega \subseteq\mathbb{R}^n$ with boundary $\Gamma$, from now on we assume one of the following conditions

\medskip
{\it (H1)  $\Omega$ is bounded and lying on one side of its topological boundary $\Gamma\in
C^\infty$};

{\it (H1bis) $\Omega = \R^n_+$};

{\it (H1ter) $\Omega= \R^n$}.

\medskip
Then, first of all, one has
$$
H^\beta(\Omega) = \{U_{|\Omega} : U \in H^\beta (\R^n)\}
$$
and an equivalent norm is $\inf\{\|U\|_{H^\beta(\R^n)}: U_{|\Omega} = u\}$.
Moreover,  $C^\infty (\overline {\Omega})$ is dense in $H^\beta(\Omega)$ (see \cite{LiMa1}, Chap. 1, Theorems 9.2, 9.3)
and it is a space of pointwise multipliers in it.

We can also define (by local charts) the spaces $H^\beta(\Gamma)$.
One can verify that, if $j \in \N_0$ and $\beta > j + \frac{1}{2}$,
the map $u \to\frac{\partial^j u}{\partial \nu^j}$, from $C^\infty (\overline \Omega)$ to $C^\infty (\Gamma)$,
can be extended to an element of $\mL (H^\beta(\Omega), H^{\beta - j - 1/2}(\Gamma))$.

If $\beta \geq 0$, we indicate with $H^\beta_0(\Omega)$ the closure of $\mD(\Omega):=C_0^\infty(\Omega)$ in  $H^\beta(\Omega)$.
It is known that, in case $\beta \leq 1/2$, $H^\beta_0(\Omega) = H^\beta(\Omega)$ (see \cite{LiMa1}, Chap. 1, Theorem 11.1).
In case $0 \leq \beta < 1/2$, the trivial extension operator with $0$ outside $\Omega$ belongs to $\mL(H^\beta(\Omega) , H^\beta(\R^n))$
(see \cite{LiMa1}, Theorem 11.4).

Now, for $\beta \geq 0$, we define
\begin{equation}\label{eq1.1A}
H^{-\beta}(\Omega):= H_0^\beta(\Omega)'.
\end{equation}
Every element $f$ of $L^2(\Omega)$ will be always identified with the functional
$$g \to \int_{\Omega} f(x) \overline{g(x)} dx$$
and, with this identification, $L^2(\Omega) \hookrightarrow H^{-\beta}(\Omega)$, $\forall \beta \geq 0$.
We observe that, as (by definition) $\mD(\Omega)$ is dense in $H_0^\beta(\Omega)$, $H^{-\beta}(\Omega)$ is a space of distributions.
One can show that, if $\beta \in \R$, $\alpha \in \N_0^n$ and $\beta - |\alpha| \not \in \{-1/2, -3/2, ...\}$, the derivative operator
$f \to D^\alpha f$ maps $H^\beta(\Omega)$ into $H^{\beta-|\alpha|}(\Omega)$ (\cite{LiMa1}, Chap. 1, Proposition 12.1).

We shall need also Sobolev spaces with values in a certain complex Hilbert spaces $V$, with scalar product $(\cdot, \cdot)_V$.
In this more general situation, we limit ourselves to consider the case $n = 1$ with $\Omega = (a,b) \subset \R, \, a <b $.
In case $\beta \geq 0$, $H^\beta(a,b; V)$ can be defined  similarly to the scalar valued situation
(see \cite{LiMa1}, Chap. 1, 2.2, \cite{LiMa2}, Chap. 4, Sec. 2.1) and the aforementioned properties can be extended
without much difficulty, employing the theory of vector valued distributions.

Further properties of Sobolev spaces will be recalled in Section \ref{seA3}.

Let $T \in \R^+$. If $\alpha, \beta \in [0, \infty)$, we set
\begin{equation}
H^{\alpha,\beta}(Q_T):= H^\alpha(0, T; L^2(\Omega)) \cap L^2(0, T; H^\beta(\Omega)).
\end{equation}
This is a Hilbert space with the norm
\begin{equation}
\|f\|_{H^{\alpha,\beta}(Q_T)}^2:= \|f\|_{H^{\alpha}(0, T; L^2(\Omega))}^2 +  \|f\|_{L^2(0, T; H^\beta(\Omega))}^2.
\end{equation}
Analogously, we define
\begin{equation}
H^{\alpha,\beta}(\Sigma_T):= H^\alpha(0, T; L^2(\Gamma)) \cap L^2(0, T; H^\beta(\Gamma)),
\end{equation}
which is a Hilbert space with the norm
\begin{equation}
\|f\|_{H^{\alpha,\beta}(\Sigma_T)}^2:= \|f\|_{H^{\alpha}(0, T; L^2(\Gamma))}^2 +  \|f\|_{L^2(0, T; H^\beta(\Gamma))}^2.
\end{equation}
We collect the following facts, which will be crucial for us:
\begin{theorem}\label{th1.1}
Assume that (H1) holds. Let $\alpha, \beta \in [0, \infty)$. Then the following propositions hold.

(I) If $ \max\{\alpha, \beta\} \leq 1/2$, then $\mD(Q_T)$ is dense in $H^{\alpha,\beta}(Q_T)$, so that its dual space  $H^{\alpha,\beta}(Q_T)'$
is a space of distributions in $Q_T$.

(II) If $j \in \N_0$ and $\beta > j + 1/2$, for $u \in H^{\alpha,\beta}(Q_T)$, $\gamma \in \N_0^n$ and $|\gamma| \leq j$,  we may define
$D_x^\gamma u_{|\Sigma_T}$, and $u \to D_x^\gamma u_{|\Sigma_T}$ is a continuous linear mapping from $H^{\alpha,\beta}(Q_T)$
to $H^{\alpha(1-\frac{j+1/2}{\beta}), \beta - j - 1/2}(\Sigma_T)$.

(III) If $k \in \N_0$ and $\alpha > k+1/2$, we may define $D_t^k u(0,\cdot)$, and $u \to D_t^k u(0,\cdot)$ is a continuous linear
mapping from $H^{\alpha,\beta}(Q_T)$ to $H^{(1-\frac{k+1/2}{\alpha})\beta}(\Omega)$.

\end{theorem}

\begin{proof} See \cite{LiMa2}, Chap. 4: for (I),  Sec. 2.1; for (II) and (III),  Theorem 2.1.
\end{proof}
We pass to outline the structure of the paper.

In Section \ref{seA2}, we state in a precise way the problem we want to solve.  The main aim of the paper is to prove a result of existence
and uniqueness of a solution to Problem 2, precisely stated in Theorem \ref{thA2.1}. The principal difficulty
in the reconstruction of the convolution kernel $h$ lies in the fact that $h$ is solution to an integral equation of the first kind,
which is a severely badly posed problem.
Following a method which, at least for parabolic systems, was introduced in \cite{LoSi1}, we differentiate the parabolic equation and the
boundary condition with respect to time and we formulate Problem 3 for the unknowns $v:=D_tu$ and $h$.  This new problem turns out to be
equivalent to Problem 2 (Proposition \ref{prA2.1}). The differentiation in time has the effect of transforming the integral equation of the
first kind into an integral equation of the second type in the unknown $h$. Moreover, it forces to look for
$u \in H^2(0, T;  L^2(\Omega)) \cap H^1(0, T;  H^2(\Omega))$, implying $v \in H^{1,2}(Q_T)$, that is the classical
functional framework for parabolic systems (see Section \ref{seA3}).

In order to solve Problem \ref{P3}, it seems very hard looking directly for a solution $v$ in $H^{1,2}(Q_T)$.
The reason is that the traces on $\Sigma_T$ of the first order space derivatives are in $H^{1/4, 1/2}(\Sigma_T)$, see \ref{th1.1} (II).
Unfortunately, in Problem 3 there appears on the boundary a term depending on memory which is not Lipschitz continuous from $H^{1,2}(Q_T)$
into $H^{1/4, 1/2}(\Sigma_T)$ and prevents us from applying the contraction mapping theorem.
To overcome this difficulty, we apply the following strategy. First we look for a weak solution $v \in H^{3/4,3/2}(Q_T)$.
Applying classical results of Lions and Magenes (see \cite{LiMa2}), we may replace the space of trace functions $H^{1/4, 1/2}(\Sigma_T)$
with $L^2(\Sigma_T)$. This allows to employ the contraction mapping theorem and obtain existence and uniqueness of a global weak solution.
The final step is to show that $v$ belongs, in fact, to $H^{1,2}(Q_T)$.

Entering into the details, in Section \ref{seA3} we revise the weak parabolic theory developed in \cite{LiMa2} and add some technical
results on it and on vector valued Sobolev spaces and their duals. Concerning dual spaces, we have adopted an abstract setting, having in mind
the space $H^{1/4,1/2}(Q_T)'$, which has a role in the weak parabolic theory.

The technical Sections \ref{seA4} and \ref{seA5} are dedicated to convolution and memory terms, respectively.
We study in particular the convolution of an element $h$ in $L^1(0, T)$ by $z$, with $z$ belonging to
a proper class of vector valued distributions in $(0, T)$, a generalization of $H^{1/4,1/2}(Q_T)'$.

In Section \ref{se6} we study the weak version of Problem \ref{P3}.
Here  we employ a method, which was introduced in \cite{CoGu1}, allowing to treat the convolution as an affine operator
(see Remark \ref{re4.4}).

Finally, in Section \ref{se7} we show that  $v \in H^{1,2}(Q_T)$ and this completes the proof of Theorem \ref{thA2.1}.

\section{Statement of the problem and equivalent formulation}\label{seA2}

\setcounter{equation}{0}

Concerning system (\ref{eqA1.15}), we assume (H1) and moreover (see \cite{LiMa1}, Chap. 2, Sec. 1)

\vskip0.2truecm
{\it (H2) $A = \sum_{i,j=1}^n D_{x_i} (a_{ij}(x) D_{x_j}), \quad  a_{ij} \in C^\infty(\overline \Omega)$
\vskip0.2truecm
\quad \quad \! $B = \sum_{i=1}^n b_{i}(x') D_{x_i} + b_0(x'), \quad b_i \in C^\infty(\Gamma);
\quad\quad \sum_{i=1}^n b_{i}(x') \nu_i(x') \neq 0$, $\forall\, x' \in \Gamma$
\vskip0.2truecm
\quad \quad \, $e^{i\theta}D_t^2 + \sum_{i,j=1}^n D_{x_i} (a_{ij}(x) D_{x_j})$
is properly elliptic in $\R \times \Omega$, and covered by $B$ in $\R \times \Gamma$,
\vskip0.1truecm
\quad \quad \, $\forall \, \theta \in [-\pi/2, \pi/2]$

\medskip

(H3) $f \in H^{1,0}(Q_T)$

\medskip
(H4) $u_0 \in H^2(\Omega)$, $v_0:= Au_0 + f(0,\cdot) \in H^1(\Omega)$

\medskip
(H5) $q \in H^{5/4}(0, T; L^2(\Gamma))\cap H^1(0, T; H^{1/2}(\Gamma))$

\medskip

(H6) $\omega \in H^2_0(\Omega)$, \quad $g \in H^2(0, T)$

\medskip

(H7) $\Phi(Au_0) \neq 0$

\medskip

(H8) $\omega_1 \in L^2(\Omega)$, \,$\omega_2 \in L^2(\Gamma)$, \, $u_A, u_B  \in H^{5/4}(0, T; L^2(\Gamma)) \cap H^1(0, T; H^{1/2}(\Gamma))$,\,

\quad \quad \, $u_C \in C([0, T]) \cap BV([0, T])$

\medskip

(H9) $\Phi(u_0) = g(0)$,  \, $\Phi(v_0) = g'(0)$, \, $Bu_0 + q(0,\cdot) =  \phi_0 u_A(0,\cdot) + u_B(0,\cdot) - u_{0|\Gamma}$ }

\medskip

\begin{remark}
{\rm In (H2) we say that  $e^{i\theta} D_t^2 + \sum_{i,j=1}^n D_{x_i}(a_{ij}(x) D_{x_j})$ is covered by $B$  when the following condition
is satisfied (see \cite{LiMa1}, Ch. 2, Prop. 4.2).

Take an  arbitrary $x'$ in $\Gamma$, $(\tau,\xi)$ in $(\R \times \R^n) \setminus \{(0,\mathbf{0})\}$ with $\xi$ tangent to
$\Gamma$ in $x'$, $\xi'$ in $\R^n \setminus \{\mathbf{0}\}$, normal to $\Gamma$ in $x'$ and consider the ODE  problem
$$
\left\{\begin{array}{l}
- e^{i\theta} \tau^2 v(t) + \sum_{i,j=1}^n a_{ij}(x') (i\xi_i + \xi_i' D_t) (i\xi_j+ \xi_j' D_t) v(t) = 0, \\ \\
\sum_{i=1}^n b_i(x') (i\xi_i + \xi_i' D_t) v (0) = 1.
\end{array}
\right.
$$
Then such problem has a unique solution $v$ which is bounded in $\R^+$.}
\end{remark}

The main result of this paper is the following

\begin{theorem}\label{thA2.1}
Assume that (C1)-(C3) and (H1)-(H9) hold. Then (\ref{eqA1.15}) has a unique solution $(u,h)$ such that
$u \in H^2(0, T; L^2(\Omega)) \cap H^1(0, T; H^2(\Omega))$ and $h \in L^2(0, T)$.
\end{theorem}

\begin{remark}
{\rm Let $u \in H^2(0, T; L^2(\Omega)) \cap H^1(0, T; H^2(\Omega))$. As $u_{|\Sigma_T} \in H^1(0, T; H^{3/2}(\Gamma))$,
then $\mM (u) \in H^1(0, T) \hookrightarrow C([0, T]) \cap BV([0, T])$. So, $\mW(\mM (u))$ is well defined and belongs to $C([0, T]) \cap BV([0, T])$.
Moreover, we deduce that $E_1 * \mW(\mM (u)) \in C^1([0, T])$ and $D_t[E_1 * \mW(\mM (u))] \in BV([0, T])$. }
\end{remark}

The first step in the proof of Theorem \ref{thA2.1} is to formulate a problem which is equivalent to Problem 2.
This is done through the following

\begin{proposition}\label{prA2.1} Assume (C1)-(C3) and (H1)-(H9). Let $(u,h)$ be a solution to (\ref{eqA1.15}) such that
\begin{equation}\label{eqA2.4}
u \in H^2(0, T; L^2(\Omega)) \cap H^1(0, T; H^2(\Omega)), \quad h \in L^2(0, T).
\end{equation}
Setting
\begin{equation}
v(t,x):= D_t u(t,x),
\end{equation}
then the pair $(v,h)$ satisfies
\begin{equation}\label{eqA2.6A}
v \in H^{1,2}(Q_T), \quad h \in L^2(0, T)
\end{equation}
and solves
\begin{problem}\label{P3}
Find $v$ of domain $Q_T$ and $h$ of domain $(0, T)$, such that
\begin{equation}\label{eqA2.6}
\left\{\begin{array}{l}
D_tv(t,x)   = Av(t,x) + h \ast Av(t,x) + v^*(t,x) \\ \\
 - [(\psi_1,v(t,\cdot)) + h \ast (\psi_1, v(t,\cdot))]Êz_0(x), \quad  (t,x) \in Q_T,  \\ \\
v(0,x) = v_0(x),  \quad x \in \Omega, \\ \\
Bv(t,y)  = - v(t,y) + \Psi(v)(t,y) - h \ast Bv(t,y) +    [(\psi_1,v(t,\cdot)) + h \ast (\psi_1, v(t,\cdot))] z_1(y)Ê\\ \\
+ v^*_\Gamma (t,y), \quad (t,y) \in \Sigma_T, \\ \\
h(t) = h^*(t) - (\psi_1,v(t,\cdot)) - h \ast (\psi_1, v(t,\cdot)), \quad t \in (0, T),
\end{array}
\right.
\end{equation}
\end{problem}
where we have set
\begin{equation}
\chi:= \Phi(Au_0)^{-1},
\end{equation}
\begin{equation}
h^*(t) := \chi(g''(t) - \Phi (D_tf(t,\cdot)), \quad t \in (0, T),
\end{equation}
\begin{equation}
z_0(x):= Au_0(x), \quad x \in \Omega,
\end{equation}
\begin{equation}
v^*(t,x) := D_tf(t,x) + h^*(t) z_0(x), \quad (t,x) \in Q_T,
\end{equation}
\begin{equation}\label{eqA2.3}
A^*:= \sum_{i,j=1}^n D_{x_j} (\overline{a_{ij}(x)}  D_{x_i}),
\end{equation}
\begin{equation}
\psi_1(x):= \chi \overline{ A^*\overline{\omega}(x)}, \quad x \in \Omega,
\end{equation}
\begin{equation}\label{eqA2.12}
\Psi(v)(t,y):= D_t [\mF (\mW(\mM(u_0 + 1 \ast v))](t,y), \quad (t,y) \in \Sigma_T,
\end{equation}
\begin{equation}
z_1(y):= Bu_0(y), \quad y \in \Gamma,
\end{equation}
\begin{equation}
v^*_\Gamma (t,y) := -D_t q(t,y) - h^*(t) z_1(y), \quad (t,y) \in \Sigma_T,
\end{equation}
\begin{equation}
(\psi_1, v):= \int_\Omega \psi_1(x) v(x) dx, \quad v \in L^2(\Omega).
\end{equation}
On the other hand, if $(v, h)$ is a solution to problem (\ref{eqA2.6}) and satisfies conditions (\ref{eqA2.6A}),
then the pair $(u, h)$ verifies (\ref{eqA2.4}) and solves system (\ref{eqA1.15}), where we have set
$u(t,x) := u_0(x) + 1 \ast v(t,x)$.

\end{proposition}

\begin{proof}
We observe that, thanks to (H1)-(H9), we have
\begin{equation}
\begin{array}{cccccc}
h^* \in L^2(0, T), & z_0 \in L^2(\Omega), & v^* \in L^2(Q_T),
&  \psi_1 \in L^2(\Omega), & z_1 \in H^{1/2}(\Gamma).
\end{array}
\end{equation}
Suppose that (\ref{eqA1.15}) has a solution $(u,h)$, with $u \in H^2(0, T; L^2(\Omega)) \cap H^1(0, T; H^2(\Omega))$, $h \in L^2(0, T)$.
Obviously, $v = D_t u \in H^{1,2}(Q_T)$. We observe that
$$
h \ast Au = h \ast (A u_0 + 1 \ast Av) = 1 \ast [h Au_0 + h \ast Av],
$$
so that $h \ast Au \in H^1(0, T; L^2(\Omega))$ and $D_t(h \ast Au) =  h Au_0 + h \ast Av$.
Analogously, $D_t(h \ast Bu) = h Bu_0 + h \ast Bv$. Hence, differentiating with respect to $t$ the two first equations in
(\ref{eqA1.15}), we obtain
\begin{equation}\label{eqA2.15}
\left\{\begin{array}{l}
D_t v(t,x) = Av(t,x) + h \ast Av(t,x) + h(t)z_0(x)  + D_tf(t,x), \quad (t,x)  \in  Q_T, \\ \\
Bv(t,y)  = -v(t,y) + D_t[\mF (\mW (\mM(u_0 + 1 \ast v)))](t,y) - h \ast Bv(t,y) \\ \\
 - h(t) z_1(y) - D_tq(t,y) , \quad (t,y)  \in  \Sigma_T, \\ \\
v(0,x) = v_0(x), \quad x \in \Omega.
\end{array}
\right.
\end{equation}
We observe also that
\begin{equation}
\begin{array}{lll}
\Phi(v(t,\cdot)) = g'(t), & \Phi(D_tv(t,\cdot)) = g''(t), & t \in (0, T).
\end{array}
\end{equation}
So, applying $\Phi$ to the first equation in (\ref{eqA2.15}), we obtain
\begin{equation}\label{eqA2.17}
h(t) = \chi[g''(t) - \Phi(Av(t,\cdot)) - h \ast \Phi(Av(t,\cdot)) - \Phi(D_t f(t, \cdot))]Ê= h^*(t) - (\psi_1,v(t, \cdot)) - h \ast (\psi_1,v(t, \cdot)),
\end{equation}
which is the last equation in (\ref{eqA2.6}). Replacing $h$ with the right term of (\ref{eqA2.17}) in (\ref{eqA2.15}), we deduce that $(v,h)$ solves
(\ref{eqA2.6}).

On the other hand, assume that $(v,h)$ satisfies (\ref{eqA2.6A}) and solves (\ref{eqA2.6}). We set $u:= u_0 + 1 \ast v$.
Then, as $u_0 \in H^2(\Omega)$, we have $u \in H^2(0, T; L^2(\Omega)) \cap H^1(0, T; H^2(\Omega))$. Moreover, (\ref{eqA2.6})
clearly implies (\ref{eqA2.15}), which can be written in the form
$$
\left\{ \begin{array}{ll}
D_t^2 u = D_t(Au + h \ast Au  + f), & \mbox{ in } Q_T, \\ \\
D_t[Bu  + h \ast Bu  + q]  = D_t[\mF (\mW (\mM u)) - u], &  \mbox{ on }  \Sigma_T.
\end{array}
\right.
$$
From (H4) and (H8),  we deduce that the first three equations in (\ref{eqA1.15}) are satisfied.
It remains only to show that $\Phi(u) \equiv g$.
Applying $\Phi$ to the first equation in (\ref{eqA2.15}) and employing (\ref{eqA2.17}), we obtain
$$
D_t^2 [\Phi(u)](t) = \Phi(D_t^2 u(t,\cdot)) = \Phi(Av(t,\cdot) + h \ast Av(t,\cdot)) + h(t)\chi^{-1}  + \Phi[D_tf(t,\cdot)] = g''(t).
$$
So the conclusion follows from (H9).
\end{proof}

\section{Weak solutions to parabolic systems}\label{seA3}

\setcounter{equation}{0}

In this section we introduce the theory of linear parabolic problems together with some further results
and remarks that we are going to use to study our inverse problem. We begin by considering a linear parabolic
system in the form

\begin{equation}\label{eq1.22}
\left\{\begin{array}{ll}
D_tu(t,x) = Au(t,x) + f(t,x), & (t,x) \in Q_T, \\ \\
u(0,x) = u_0(x), & x \in \Omega, \\ \\
Bu(t,x') = g(t,x'), & (t,x') \in \Sigma_T,
\end{array}
\right.
\end{equation}
We outline now some topics of the theory contained in Chap. 4 of \cite{LiMa2}.
Recalling the definition of the adjoint operator $A^*$ of $A$ (see (\ref{eqA2.3})),
the first result is concerned with the elliptic theory.

\begin{theorem}\label{th1.2}
Assume that (H1)-(H2) hold. Then we can construct three operators $S$, $C$, $T$, such that:

(I) $S$ and $T$ are multiplication operators by functions $s(x')$ and $t(x')$ in $C^\infty (\Gamma)$,
such that $s(x') \neq 0$ and $t(x') \neq 0$, $\forall x' \in \Gamma$;

(II) $C$ is a first order differential operator, with coefficients in $C^\infty (\Gamma)$;

(III) (H2) holds with $A$ replaced by $A^*$ and $B$ replaced by $C$;

(IV) $\forall \, u, v \in H^2(\Omega)$, the following Green's formula holds
\begin{equation}\label{eq1.24}
\int_\Omega [Au(x) \overline{v(x)} - u(x) \overline{A^*v(x)}] dx = \int_{\Gamma}Ê[Su(x') \overline{Cv(x')} - Bu(x') \overline{Tv(x')}] d\sigma.
\end{equation}

\end{theorem}

\begin{proof} See \cite{LiMa1}, Chap. 2.
\end{proof}

From (\ref{eq1.24})  we immediately deduce the following formula, valid for $u, v \in H^{1,2}(Q_T)$,
\begin{equation}
\begin{array}{c}
\int_{Q_T} \{[D_tu(t,x) - Au(t,x)] \overline{v(t,x)} + u(t,x) [\overline{D_tv(t,x) + A^*v(t,x)}]\} dt dx = \\ \\
= \int_\Omega [u(T,x) \overline{v(T,x)} -  u(0,x) \overline{v(0,x)}] dx +
\int_{\Sigma_T}Ê[ Bu(t,x') \overline{Tv(t,x')} - Su(t,x') \overline{Cv(t,x')}] dt d\sigma.
\end{array}
\end{equation}

The next fundamental result holds.

\begin{theorem}\label{th1.3}
Assume (H1)-(H2). Then (\ref{eq1.22}) admits a unique solution $u \in H^{1,2}(Q_T)$
if and only if
$$
\begin{array}{ccc}
f \in L^2(Q_T), & u_0 \in H^1(\Omega), & g \in H^{1/4,1/2}(\Sigma_T).
\end{array}
$$
\end{theorem}
\begin{proof} See \cite{LiMa2}, Chap. 4, Theorem 5.3.
\end{proof}
\begin{remark}
{\rm If  $u \in H^{1,2}(Q_T)$ is the solution of (\ref{eq1.22}), then it holds
\begin{equation}
\begin{array}{c}
\int_{Q_T}  u(t,x) [\overline{D_tv(t,x) + A^*v(t,x)}] dt dx = \\ \\
= - \int_{Q_T} f(t,x) \overline{v(t,x)} dt dx -  \int_\Omega  u_0(x) \overline{v(0,x)}] dx
+  \int_{\Sigma_T}Êg(t,x') \overline{Tv(t,x')} dt d\sigma, \\ \\
\forall \, v \in H^{1,2}(Q_T) : Cv \equiv 0, v(T,\cdot) = 0.
\end{array}
\end{equation}
}
\end{remark}

A simple consequence of Theorems \ref{th1.2} Êand  \ref{th1.3}  is the following

\begin{corollary}\label{co1.1}
Assume that (H1)-(H2) are fulfilled and consider the system
\begin{equation}\label{eq1.26}
\left\{\begin{array}{ll}
D_tv(t,x) + A^*v(t,x) = \phi(t,x), & (t,x) \in Q_T, \\ \\
v(T,x) = 0, & x \in \Omega, \\ \\
Cv(t,x') = 0, & (t,x') \in \Sigma_T.
\end{array}
\right.
\end{equation}
If $\phi \in L^2(Q_T)$, then (\ref{eq1.26}) has a unique solution $v \in H^{1,2}(Q_T)$.
 \end{corollary}

From property  (II) of Theorem \ref{th1.1}, it follows that, if $v \in H^{1,2}(Q_T)$, then $Tv \in H^{3/4,3/2}(Q_T)$.
So, employing a simple duality argument, from Corollary \ref{co1.1} we deduce the following

 \begin{corollary}\label{co1.2}
 Suppose that (H1)-(H2) are fulfilled. Then, $\forall f \in H^{1,2}(Q_T)'$, $\forall \, u_0 \in H^1(\Omega)'$,
 $\forall \, g \in H^{3/4,3/2}(Q_T)'$, there exists a unique $u = S_T(f, u_0,g)$ in $L^2(Q_T)$ such that
 \begin{equation}
\begin{array}{c}
\int_{Q_T}  u(t,x) [\overline{D_tv(t,x) + A^*v(t,x)}] dt dx =
 - (f, v) -  (u_0, v(0,\cdot)) + (g,  Tv), \\ \\
 \forall \, v \in H^{1,2}(Q_T) : Cv \equiv 0, v(T,\cdot) = 0.
\end{array}
\end{equation}
Moreover, if $f \in L^2(Q_T)$, $u_0 \in H^1(\Omega)$ and $g \in H^{1/4,1/2}(\Sigma_T)$, then $u \in H^{1,2}(Q_T)$ and
it solves  (\ref{eq1.22}).
 \end{corollary}

 \begin{proof} See \cite{LiMa2}, Chap. 4, Sec. 13.3.
 \end{proof}
 \begin{remark}
 {\rm  An $L^p$ version of Corollary \ref{co1.2} ($1 < p < \infty$) is given in \cite{Am1}, Theorem 0.5. Here the author
 defines $S_T(f, u_0,g)$ as {\it ultraweak solution} of the parabolic problem with data $(f, g,u_0)$.}
 \end{remark}

 \medskip

Another key tool for our analysis is
 \begin{theorem}\label{th1.4}
  Assume (H1)-(H2). Then $S_T$ maps $H^{1/4,1/2}(Q_T)' \times H^{1/2}(\Omega) \times L^2(\Sigma_T)$ into $H^{3/4, 3/2}(Q_T)$.
 \end{theorem}
\begin{proof} See \cite{LiMa2}, Chap. 4, Secs. 15.1 and 15.3.
 \end{proof}
Now we want to derive appropriate estimates of $S_T(f, u_0,g)$. To this aim, we consider in detail the space $H^\beta(0,T; V)$
with $\beta \in (0, 1)$ and $V$ Hilbert. It can be characterized as
$$
\left\{f \in L^2(0, T; V) \,\,\, \text{such that} \,\,\,  [f]_{H^\beta(0, T; V)}^2:=
\int_0^T \left(\int_0^t  \frac{\|f(t) - f(s)\|^2_V}{(t-s)^{1+2\beta}} ds \right) dt < \infty\right\}
$$
(see, for example, \cite{Ad1}, Theorem 7.48, or \cite{Am1}, Theorem 4.4.3) and it is a Hilbert space with the norm
\begin{equation}\label{norm1}
 \left (\|f\|'_{H^\beta(0, T; V)}\right )^2 := \|f\|_{L^2(0, T; V)}^2 + [f]^2_{H^\beta(0, T; V)}.
\end{equation}
Consider first the case $\beta \in (0, 1/2)$. One can prove that there exists $C>0$ such that, $\forall f \in H^\beta(0, T; V)$,
\begin{equation}\label{eq1.3}Ê
\int_0^T [t^{-2\beta} + (T-t)^{-2\beta}]Ê\|f(t)\|_V^2 dt \leq C \left( \|f\|'_{H^\beta(0, T; V)}\right )^2.
\end{equation}
(see \cite{LiMa1}, Chap. 1, Sec. 11.2). On the other hand, if $\beta \in (1/2, 1)$, then $H^\beta(0, T; V) \hookrightarrow
C^{\beta-1/2}([0, T]; V)$, the space of H\"older functions with values in $V$ (see  \cite{Si1}, Chap. 14).
So, introducing the norms
\begin{equation}\label{eq1.1}
\|f\|_{H^\beta(0, T; V)}^2Ê:=
\left\{ \begin{array}{lll}
 \int_0^T t^{-2\beta} \|f(t)\|_V^2 dt + \int_0^T (\int_0^t  \frac{\|f(t) - f(s)\|^2_V}{(t-s)^{1+2\beta}} ds ) dt, & \mbox{Êif }Ê 0 < \beta < 1/2, \\ \\
 \|f(0)\|_V^2 + \int_0^T (\int_0^t  \frac{\|f(t) - f(s)\|^2_V}{(t-s)^{1+2\beta}} ds ) dt, & \mbox{Ê if }Ê 1/2 < \beta < 1,
 \end{array}
\right.
\end{equation}
then it holds
\begin{proposition}
Let $\beta \in (0, 1) \setminus \{1/2\}$. Then the norms defined in (\ref{norm1}) and in (\ref{eq1.1}) are equivalent.
\end{proposition}

\begin{proof} The case $\beta \in (0, 1/2)$ follows easily from (\ref{eq1.3}). Concerning the case $\beta \in (1/2, 1)$,
it suffices to prove the estimate
$$
\|f\|_{L^2(0, T; V)}^2 \leq C \|f\|_{H^\beta(0, T; V)}^2, \, \forall \, f \in H^\beta(0, T; V),
$$
for some $C > 0$, independent of $f$. If this is not the case, then there exists a sequence $(f_k)_{k \in \N}$ in $H^\beta(0, T; V)$,
such that $\|f_k\|_{L^2(0, T; V)} = 1$ for every $k \in \N$, while  ${\displaystyle \lim_{k \to \infty}}Ê\|f_k\|_{H^\beta(0, T; V)} = 0$.
In particular,
$$ \lim_{k \to \infty} \int_0^T \left (\int_0^T \frac{\|f_k(t) - f_k(s)\|^2_V}{|t-s|^{1+2\beta}} ds \right ) dt
\leq 2 \|f_k\|_{H^\beta(0, T; V)} \to 0, \quad  (k \to \infty). $$
This  implies that, possibly passing to a subsequence, we may assume, for almost every $t$ in $[0, T]$,
\begin{equation}\label{eq1.5}
 \int_0^T  \frac{\|f_k(t) - f_k(s)\|^2_V}{|t-s|^{1+2\beta}} ds \to 0, \quad (k \to \infty).
\end{equation}
By an application of imbedding Sobolev theorems, there  exists $C_1>0$, such that, $\forall \, k \in \N$, $\forall \, t \in [0, T]$,
$$
\|f_k(t) - f_k(0)\|_V \leq C_1 {\|f_k\|'_{H^\beta(0, T; V)}} t^{\beta - 1/2}Ê\leq C_2 t^{\beta - 1/2}.
$$
So, as ${\displaystyle \lim_{k \to \infty}} \|f_k(0)\|_V = 0$, for every $\epsilon>0$ we can choose $t \in [0, T]$, such that
$\|f_k(t)\|_V \leq \epsilon$, $\forall \, k \in \N$, and (\ref{eq1.5}) holds. Hence, we can deduce
$$
1 = \int_0^T \|f_k(s)\|_V^2 ds
\leq 2\left( T^{1+2\beta}Ê\int_0^T\frac{\|f_k(t) - f_k(s)\|^2_V}{|t-s|^{1+2\beta}} ds + T\epsilon^2\right)
\leq 2( T^{1+2\beta}Ê + T)\epsilon^2,
$$
if $k$ is sufficiently large, which is clearly a contradiction.
\end{proof}

We observe that the norm $\|\cdot\|_{H^\beta(0, T; V)}$  will be particularly convenient in the case $0 < \beta < \frac{1}{2}$,
in force of the homogeneity property
\begin{equation}\label{eq1.2}
\|f\|_{H^\beta(0, T; V)}^2 = T^{1-2\beta} \|f(T\cdot)\|_{H^\beta(0, 1; V)}^2 \quad \forall \, T >0.
\end{equation}
For future use, we prove the following
\begin{lemma}\label{le1.1A}
Let $\beta \in (1/2, 1)$. Then there exists $C >0$, independent of $T$, such that
$$
\|f\|_{L^2(0, T; V)}^2
\leq CT\left (\|f(0)\|_V^2 + T^{2\beta -1} \int_0^T \left(\int_0^t \frac{\|f(t) - f(s)\|_V^2}{(t-s)^{1+2\beta}} ds\right) dt\right ),
\quad  \forall \, f \in H^\beta(0, T; V).
$$
\end{lemma}
\begin{proof} It follows from the chain of inequalities
$$
\begin{array}{c}
\|f\|_{L^2(0, T; V)}^2 = T \| f(T\cdot)\|_{L^2(0, 1; V)}^2 \leq CT \|f(T\cdot)\|_{H^\beta(0, 1; V)}^2 \\ \\
= CT(\|f(0)\|_V^2 + \int_0^1 (\int_0^t \frac{\|f(Tt) - f(Ts)\|_V^2}{(t-s)^{1+2\beta}} ds) dt) \\ \\
= CT(\|f(0)\|_V^2 + T^{2\beta -1} \int_0^T (\int_0^t \frac{\|f(t) - f(s)\|_V^2}{(t-s)^{1+2\beta}} ds) dt).
\end{array}
$$
\end{proof}

\begin{lemma}\label{le1.1}
 Let $\beta \in (0, 1/2)$. Then the following propositions hold.

(I) There exists $C>0$, independent of $T$, such that, $\forall \, f \in H^\beta(0, T; V)$,
$$
\int_0^T (T-t)^{-2\beta} \|f(t)\|_V^2 dt \leq C \|f\|_{H^\beta(0, T; V)}^2.
$$
(II) Let $0 < T < T'$. Given $f$ in $H^\beta(0, T; V)$, we indicate with $\widetilde f$ its extension to $(0, T')$,
such that $\widetilde f(t) = 0$ for a.a. $t \in [T, T')$. Then $\widetilde f \in H^\beta(0, T'; V)$. Moreover, there exists
$C >0$, independent of $T, T', f$, such that
$$
\|\widetilde f \|_{H^\beta(0, T'; V)}^2 \leq C \|f \|_{H^\beta(0, T; V)}^2.
$$
(III) Given $f$ in $H^\beta(0, T; V)$ and $s \in (0, T)$, we define, for $t \in (0, T)$,
\begin{equation}\label{eq1.4}
f_s(t) = \left\{\begin{array}{ll}
f(t+s), & \mbox{ if } t + s < T, \\ \\
0, & \mbox{ if } t + s \geq T.
\end{array}
\right.
\end{equation}
Then $f_s \in H^\beta(0, T; V)$. Moreover, the map $s \to f_s$ belongs to $C((0, T);  H^\beta(0, T; V))$ and
$$
\| f_s \|_{H^\beta(0, T; V)}^2 \leq C \|f \|_{H^\beta(0, T; V)}^2,
$$
with $C>0$, independent of $T$, $s$, $f$.
\end{lemma}

\begin{proof} (I) follows immediately from (\ref{eq1.2}). In fact,
$$
\begin{array}{c}
\int_0^T (T-t)^{-2\beta} \|f(t)\|_V^2 dt = T^{1-2\beta} \int_0^1 (1-t)^{-2\beta} \|f(Tt)\|_V^2 dt \\ \\
 \leq C T^{1-2\beta} \|f(T\cdot)\|_{H^\beta(0, 1; V)}^2 = C  \|f\|_{H^\beta(0, T; V)}^2.
\end{array}
$$
Concerning (II), we have
$$
\begin{array}{c}
\|\widetilde f\|_{H^\beta(0, T'; V)}^2 = \|f\|_{H^\beta(0, T; V)}^2 + \int_T^{T'} (\int_0^T \frac{\|f(s)\|_V^2}{(t-s)^{1+2\beta}} ds ) dt
\\ \\
\leq \|f\|_{H^\beta(0, T; V)}^2 + \int_0^T (\int_T^\infty (t-s)^{-1-2\beta} dt) \|f(s)\|_V^2 ds \\ \\
=  \|f\|_{H^\beta(0, T; V)}^2 + \frac{1}{2\beta}Ê\int_0^T (T-s)^{-2\beta}Ê\|f(s)\|_V^2 ds,
\end{array}
$$
and the conclusion follows from (I).

In order to prove (III),  we start by considering the case $T = 1$. We adopt the norm
$$
\|f\|''_{H^\beta(0, 1; V)}:= \inf\{\|F\|_{H^\beta(\R; V)}: F_{|(0, 1)} = f\},
$$
with
$$
\|F\|_{H^\beta(\R; V)}^2 := \int_\R (1 + \tau^2)^\beta \|\widehat F(\tau)\|_V^2 d\tau
$$
and $\widehat F$ the Fourier transform of $F$.
It is easy to see that $\|F(\cdot + s)\|_{H^\beta(\R; V)} = \|F\|_{H^\beta(\R; V)}$, $\forall \, s \in \R$,
and the map $s \to F(\cdot + s)$ is continuous from $\R$ to $H^\beta(\R; V)$.
As the trivial extension of an element of $H^\beta(0, 1; V)$ is an element of $H^\beta(\R; V)$, we may think
of the characteristic function $\chi$ of $(0, 1)$ as a pointwise multiplier in $H^\beta(\R; V)$.
So, if $F \in  H^\beta(\R; V)$ and $F_{|(0, 1)} = f$, we have
$$
\|f_s\|''_{H^\beta(0, 1; V)} \leq \|\chi F(\cdot + s)\|_{H^\beta(\R; V)} \leq C \|F(\cdot + s)\|_{H^\beta(\R; V)}
= C \|F\|_{H^\beta(\R; V)},
$$
implying
$$
\|f_s\|''_{H^\beta(0, 1; V)} \leq  C \|f\|''_{H^\beta(0, 1; V)},
$$
for some $C >0$, independent of $s \in (0, 1)$ and $f$. Moreover, if $s_k \to s$,
$$
\|f_{s_k} - f_s \|''_{H^\beta(0, 1; V)} \leq C\|F(\cdot + s_k) - F(\cdot + s)\|_{H^\beta(\R; V)} \to 0, \quad {\text as} \, \, k \to \infty.
$$
 The general case follows from (\ref{eq1.2}), since if $s \in (0, T)$ and $f \in H^\beta(0, T; V)$ then we have
$$
\begin{array}{c}
\|f_s\|_{H^\beta(0, T; V)}^2 = T^{1-2\beta} \|f_s(T\cdot)\|_{H^\beta(0, 1; V)}^2 = T^{1-2\beta} \|[f(T\cdot)]_{s/T}\|_{H^\beta(0, 1; V)}^2 \\ \\
\leq C T^{1-2\beta}  \|f(T\cdot)\|_{H^\beta(0, 1; V)}^2 = C \|f\|_{H^\beta(0, T; V)}^2.
\end{array}
$$
\end{proof}

As we have in mind to apply Theorem \ref{th1.4}, we have to deal with the space $H^{1/4,1/2}(Q_T)'$.
So we consider the following abstract framework.

\medskip

\begin{definition}\label{alfa} Let $V_1$ and $V_2$ be complex Hilbert spaces with $V_1$ densely embedded into $V_2$.
For $0 < \beta < 1/2$, we define the Hilbert space
$$
Y_T:= H^\beta(0, T; V_2) \cap L^2(0, T; V_1),
$$
equipped with the norm
$$
\|f\|_{Y_T}^2 := \|f\|_{H^\beta(0, T; V_2)}^2 + \|f\|_{L^2(0, T; V_1)}^2
$$
and its antidual space $Y_T' = [H^\beta(0, T; V_2) \cap L^2(0, T; V_1)]'$, normed by
$$
\|z\|_{Y_T'} := \sup \{|(z,f)| \, : \, \|f\|_{Y_T} \leq 1\}.
$$
\end{definition}
\medskip
Observe that we shall think of $L^2(0, T; V_2)$ as continuously embedded into $Y_T'$, identifying $g \in L^2(0, T; V_2)$ with the functional
$$
f \to \int_0^T (g(t), f(t))_{V_2} dt.
$$
The first important fact is

\begin{lemma}\label{le1.2}
 $L^2(0, T; V_2)$ is dense in $Y_T'$.
\end{lemma}

\begin{proof} Owing to the reflexivity of $Y_T$, it suffices to show that, if $f \in Y_T$, and $\int_0^T (g(t), f(t))_{V_2} dt = 0$,
$\forall \,g \in L^2(0, T; V_2)$, then $f(t) = 0$ a.e. in $(0,T)$, which is obvious.
\end{proof}

Let $0 < T < T'$. Given $z \in Y_{T'}'$ and recalling the definition of $\widetilde f$, we can define the restriction $z_{|(0, T)}$ of
$z$ to $(0, T)$ as the element of $Y_T'$, such that
\begin{equation}\label{eq1.6}
(z_{|(0, T)}, f):= (z, \widetilde f), \quad f \in Y_T.
\end{equation}

On the other hand, taking $z \in Y_{T}'$, we can define its trivial extension $\widetilde z$ to $(0, T')$ as
\begin{equation}\label{eq1.7A}
(\widetilde z, f):= (z, f_{|(0, T)}), \quad f \in Y_{T'}.
\end{equation}
We observe that these definitions coincide with the natural ones in $L^2(0, T'; V_2)$ and $L^2(0, T; V_2)$.

\begin{lemma} \label{le1.3}
The following propositions hold.

 (I) There exists $C >0$, independent of $T$ and $T'$, such that
\, $\|z_{|(0, T)}\|_{Y_T'}  \leq  C \|z\|_{Y_{T'}'}, \forall z \in Y_{T'}'$.

(II) $\|\widetilde z\|_{Y_{T'}'}  \leq   \|z\|_{Y_{T}'}$, $\forall \,z \in Y_{T}'$.

(III) If \, $g \in L^2(0, T; V_2)$, then \,  $\|g\|_{Y_T'} \leq T^\beta \|g\|_{L^2(0, T; V_2)}$.
\end{lemma}

\begin{proof}
(I) By Lemma \ref{le1.1}, there exists $C$, independent of $T$ and $T'$, such that
$$
|(z_{|(0, T)}, f)| = |(z,  \widetilde f)| \leq \|z\|_{Y_{T'}'} \|\widetilde f\|_{Y_{T'}} \leq C \|z\|_{Y_{T'}'} \|f\|_{Y_{T}}, \quad \forall \, f \in Y_T.
$$
(II)  $\forall \, f \in Y_{T'}$ then it holds
$$
|(\widetilde z, f)| = |(z, f_{|(0, T)})| \leq \|z\|_{Y_{T}'} \|f_{|(0, T)}\|_{Y_T} \leq \|z\|_{Y_{T}'} \|f\|_{Y_{T'}}.
$$
(III) If $f \in Y_T$, we have
$$
\begin{array}{c}
|(g, f)| = |\int_0^T (g(t), f(t))_{V_2} dt| \leq T^\beta \int_0^T \|g(t)\|_{V_2} t^{-\beta} \|f(t)\| dt \\ \\
 \leq  T^\beta \|g\|_{L^2(0, T; V_2)}Ê(\int_0^T t^{-2\beta} \|f(t)\|_{V_2}^2 dt)^{1/2} \leq T^\beta \|g\|_{L^2(0, T; V_2)} \|f\|_{Y_T}.
\end{array}
$$
\end{proof}

\begin{remark}
{\rm   If $z \in Y_{T'}'$ and $0 < T < T'$, we get
\begin{equation}\label{eq1.7}
(z, \chi_{(0, T)} f) = (z, {\widetilde f}_{|(0, T)}) = (z_{|(0, T)}, f_{|(0, T)}), \quad \forall \, f \in Y_{T'}.
\end{equation}
We can also define $z_{|(T, T')}$, which will be an element of $[H^\beta(T, T'; V_2) \cap L^2(T, T';  V_1)]'$.
Observe that $H^\beta(T, T'; V_2) \cap L^2(T, T';  V_1)$ is a Hilbert space, with the norm
\begin{equation}
\begin{array}{c}
\|f\|_{H^\beta(T, T'; V_2) \cap L^2(T, T';  V_1)}^2 \\ \\
:= \int_T^{T'}Ê[(t-T)^{-2\beta} \|f(t)\|_{V_2}^2 + \|f(t)\|_{V_1}^2 + \int_T^t \frac{\|f(t) - f(s)\|_{V_2}^2}{(t-s)^{1+2\beta}} ds] dt.
\end{array}
\end{equation}
Let $f \in H^\beta(T, T'; V_2) \cap L^2(T, T';  V_1)$. We can consider the element $f_0 \in Y_{T'}$, which is the trivial extension of $f$ to $(0, T')$:
\begin{equation}
f_0(t):= \left\{\begin{array}{ll}
0 & \mbox{ if }Êt \in (0, T], \\ \\
f(t) &  \mbox{ if }Êt \in (T, T'].
\end{array}
\right.
\end{equation}
and set
\begin{equation}\label{eq1.11A}
(z_{|(T, T')}, f) := (z, f_0).
\end{equation}
One can see that, if $f \in Y_{T'}$,
\begin{equation}\label{eq1.11}
(z, \chi_{(T, T')} f) = (z_{|(T, T')}, f_{|(T, T')}).
\end{equation}
We can associate with any element $v \in [H^\beta(T, T'; V_2) \cap L^2(T, T';  V_1)]'$ an element $v(\cdot + T) \in Y_{T'-T}'$ in the following way:
\begin{equation}\label{eq1.12}
(v(\cdot + T), f):= (v, f(\cdot  - T)), \quad f \in Y_{T'-T}.
\end{equation}
Clearly, $v \to v(\cdot + T)$ is an isometry between $[H^\beta(T, T'; V_2) \cap L^2(T, T';  V_1)]'$ and $Y_{T'-T}'$.
}
\end{remark}

\medskip
\begin{lemma}\label{le3.14}
Let $0 < T < T'$. Then the map $z \to (z_{|(0, T)}, z_{|(T, T')})$ is a bicontinuous bijection between $Y_{T'}'$ and
$Y_T' \times  [H^\beta(T, T'; V_2) \cap L^2(T, T';  V_1)]'$.
\end{lemma}

\begin{proof} Let $z_0 \in Y_T' $ and let $z_1 \in [H^\beta(T, T'; V_2) \cap L^2(T, T';  V_1)]'$. Then, if $z \in Y_{T'}'$, $z_{|(0, T)} = z_0$,
$ z_{|(T, T')} = z_1$, by (\ref{eq1.7}) and (\ref{eq1.11}) we have
\begin{equation}\label{eq1.13}
(z, f) = (z, \chi_{(0, T)} f) +  (z, \chi_{(T, T')} f) = (z_0, f_{|(0, T)}) + (z_1, f_{|(T, T')}), \quad  \forall \, f \in Y_{T'}.
\end{equation}
On the other hand, the right term in (\ref{eq1.13}) defines an element $z$ of $Y_{T'}'$ such that $z_{|(0, T)} = z_0$, $ z_{|(T, T')} = z_1$.
The bicontinuity follows from Lemma \ref{le1.3} (I)-(II).
\end{proof}

Let us go back now to parabolic problems. We have

 \begin{proposition}\label{pr1.1}
 Assume (H1)-(H2) and let $S_T$ be the operator defined in the statement of Corollary \ref{co1.2}. Then there hold

 (I) the restriction of $S_T$ to $H^{1/4,1/2}(Q_T)' \times H^{1/2}(\Omega) \times L^2(\Sigma_T)$ is injective;

 (II) if $u = S_T(f,u_0,g)$, with $(f,u_0,g) \in H^{1/4,1/2}(Q_T)' \times H^{1/2}(\Omega) \times L^2(\Sigma_T)$, then
 $$
 f = D_t u - Au
 $$
 in the sense of distributions and
 $$u_0 = u(0,\cdot).$$
 \end{proposition}

\begin{proof} (I) Let $(f,u_0,g) \in H^{1/4,1/2}(Q_T)' \times H^{1/2}(\Omega) \times L^2(\Sigma_T)$ be such that $S_T(f,u_0,g) = 0$, that is,
 $$
 - (f, v) -  \int_\Omega u_0(x) \overline{v(0, x)} dx + \int_{\Sigma_T}Êg(t,x') \overline{t(x') v(t,x')} dt d\sigma = 0
 $$
 $$\forall v \in H_T:=\left\{v \in H^{1,2}(Q_T): Cv = 0, v(T,\cdot) = 0\right\}. $$
 Taking $v \in \mathcal D(Q_T)$, we obtain $(f,v) = 0$. As $\mathcal D(Q_T)$ is dense in $H^{1/4,1/2}(Q_T)$ (cf. Theorem \ref{th1.1} (I)),
 we obtain $f = 0$. Next, let us fix $v_0 \in \mD(\Omega)$ and take
 $$
 v(t,x):= \zeta(t) v_0(x),
 $$
 with $\zeta \in C^\infty([0, T])$, such that $\zeta(0) = 1$ and $\zeta(T) = 0$. We deduce
 $$
 \int_\Omega u_0(x) \overline{v_0(x)} dx = 0, \quad \forall \, v_0 \in \mD(\Omega),
 $$
 implying $u_0 = 0$. Hence we conclude that
 $$\int_{\Sigma_T}Êg(t,x') \overline{t(x') v(t,x')} dt d\sigma = 0, \quad \forall \,v \in H_T,$$
 which implies, as $t(x') \neq 0$, $\forall \, x' \in \Gamma$, that $g = 0$.

 (II) Taking $v \in \mD(Q_T)$ we have
 $$
 \int_{Q_T}  u(t,x) [\overline{D_tv(t,x) + A^*v(t,x)}] dt dx =  - (f, v),
 $$
and then $D_tu - Au = f$ (we recall that $H^{1/4,1/2}(Q_T)'$ is a space of distributions). Next, we consider
a sequence $(f_k, u^k_0, g_k)$ in $L^2(Q_T) \times H^1(\Omega) \times H^{1/4,1/2}(\Sigma_T)$, converging
to $(f, u_0, g)$ in  $H^{1/4,1/2}(Q_T)' \times H^{1/2}(\Omega) \times L^2(\Sigma_T)$  (we employ Lemma \ref{le1.2}
to show the existence of such a sequence). Setting $u_k:= S_T(f_k, u^k_0, g_k)$, then the sequence $(u_k)_{k\in \N}$ converges to
$u \in H^{3/4,3/2}(Q_T)$,  so that, by Theorem \ref{th1.1} (III), we infer $(u^k_0)_{k \in \N} = (u_k(0,\cdot))_{k \in \N}$ converging
to $u(0,\cdot)$ in $H^{1/2}(\Omega)$. And finally we get $u_0 = u(0, \cdot)$.
 \end{proof}

\begin{remark}\label{re1.3}
 {\rm  On account of the Hahn-Banach theorem, it is possible to identify the elements of $H^{1/4,1/2}(Q_T)'$
 with the distributions $z$ in $Q_T$, which can be represented (in not unique way) in the form
 $$
 z = z_0 + z_1,
 $$
with $z_0 \in H^{1/4,0}(Q_T)' = H^{-1/4}(0, T; L^2(\Omega))$ and $z_1 \in L^2(0, T; H^{-1/2}(\Omega))$ (see \cite{BeLo1}, Theorem 2.7.1).
Observe that if $u \in H^{3/4,3/2}(Q_T)$, then $D_tu \in H^{-1/4}((0, T); L^2(\Omega))$ (see Proposition 12.1 in \cite{LiMa1}, Ch. I,
extended to the vector valued case, and also \cite{Am1}, Theorem 4.4.2).
However, as (in general) if $z \in H^{3/2}(\Omega)$ it does not hold  $Az \in H^{-1/2}(\Omega)$,
then we cannot infer as well that $Au \in H^{1/4,1/2}(Q_T)'$.
On the other hand, we deduce by difference from Proposition \ref{pr1.1} (II)
that, if $u = S_T(f,u_0,g)$, with $(f,u_0,g) \in H^{1/4,1/2}(Q_T)' \times H^{1/2}(\Omega) \times L^2(\Sigma_T)$,
then necessarily $Au \in H^{1/4,1/2}(Q_T)'$.

We observe also that, if $u \in H^{3/4,3/2}(Q_T)$, then $u_{|\Sigma_T} \in H^{1/2,1}(\Sigma_T) \hookrightarrow L^p(0, T; L^2(\Gamma))$,
$\forall p \in [1, \infty)$ (see \cite{Si1}, Lemma 8).}
\end{remark}

\medskip
We introduce now the following functional space
\begin{definition}\label{de3.1}
 Assume (H1)-(H2). We indicate with $X_T$ the range of $S_T$, restricted to $H^{1/4,1/2}(Q_T)' \times H^{1/2}(\Omega) \times L^2(\Sigma_T)$.
 If $u \in X_T$, then $Au \in H^{1/4,1/2}(Q_T)'$ (cf. Remark \ref{re1.3}). Moreover, by Proposition \ref{pr1.1} (I), the element
 $Bu:= g \in L^2(\Sigma_T)$ is uniquely determined. Hence, for a fixed $p \in (2, \infty)$, we can introduce in $X_T$ the norm
\begin{equation}\label{eq1.30}
 \|u\|_{X_T} := \|u\|_{H^{3/4,3/2}(Q_T)} + \|u_{|\Sigma_T}\|_{L^p(0, T; L^2(\Gamma))} + \|Au\|_{H^{1/4,1/2}(Q_T)'} + \|Bu\|_{L_2(\Sigma_T)}.
 \end{equation}
  \end{definition}
 \begin{lemma}
 Assume (H1)-(H2). Then the following propositions hold.

  (I) Let $u \in H^{3/4,3/2}(Q_T)$. Then $u \in X_T$ if and only if $Au \in H^{1/4,1/2}(Q_T)'$ and there exists a sequence
  $(u_k)_{k \in \N}$ in $H^{1,2}(Q_T)$, such that
  \begin{equation}\label{eq1.31}
  \|u_k - u\|_{H^{3/4,3/2}(Q_T)}^2 +  \|Au_k - Au\|_{H^{1/4,1/2}(Q_T)'}^2 \to 0 \quad (k \to \infty)
  \end{equation}
  with $(Bu_k)_{k \in \N}$ converging in $L^2(\Sigma_T)$.

  (II) $X_T$ is a Banach space.
 \end{lemma}

 \begin{proof} (I) Let $u = S_T(f,u_0,g) \in X_T$ and take $(f_k, u^k_0, g_k) \in L^2(Q_T) \times H^1(\Omega) \times H^{1/4,1/2}(\Sigma_T)$ such that
 $$
 \|f_k - f\|_{H^{1/4,1/2}(Q_T)'}^2 + \|u^k_0 - u_0\|_{H^{1/2}(\Omega)}^2 + \|g_k - g\|_{L^2(\Sigma_T)}^2 \to 0 \quad (k \to \infty).
 $$
 Setting  $u_k:= S_T(f_k, u^k_0, g_k)$, then $u_k \in H^{1,2}(Q_T)$ and $\|u_k - u\|_{H^{3/4,3/2}(Q_T)}^2 \to 0$ $(k \to \infty)$ from which
 $$\|D_tu_k - D_tu\|_{H^{-1/4}((0, T); L^2(\Omega))} \to 0 \quad (k \to \infty).$$
 By difference
 $$
 Au_k = -f_k + D_tu_k \to -f + D_tu = Au \quad (k \to \infty)
 $$
 in $H^{1/4,1/2}(Q_T)'$. Moreover, $Bu_k = g_k \to g$ in $L^2(\Sigma_T)$.

 On the other hand, let $u \in H^{3/4,3/2}(Q_T)$ be such that $Au \in H^{1/4,1/2}(Q_T)'$ and assume  that there exists a sequence
 $(u_k)_{k \in \N} \in H^{1,2}(Q_T)$ satisfying (\ref{eq1.31}) with  $\|Bu_k - g\|_{L^2(\Sigma_T)}Ê\to 0$ ($k \to \infty$), for some $g \in L^2(\Sigma_T)$.
 Consider now $v \in H^{1,2}(Q_T)$, with $Cv = 0$ and $v(T,\cdot) = 0$. Then, $\forall \, k \in \N$,
 \begin{equation}\label{eq1.32}
 \begin{array}{c}
 \int_{Q_T}  u_k(t,x) [\overline{D_tv(t,x) + A^*v(t,x)}] dt dx = \\ \\
= - \int_{Q_T} (D_tu_k(t,x)  - Au_k(t,x)) \overline{v(t,x)} dt dx -  \int_\Omega  u_k(0,x) \overline{v(0,x)}] dx \\ \\
+  \int_{\Sigma_T}ÊBu_k(t,x') \overline{Tv(t,x')} dt d\sigma.
\end{array}
\end{equation}
From (\ref{eq1.31}) we have that
$$
\|D_tu_k - D_tu\|_{H^{-1/4}(0, T; L^2(\Omega))} + \|u_k(0,\cdot) - u(0,\cdot)\|_{H^{1/2}(\Omega)}Ê\to 0 \quad (k \to \infty).
$$
So we can pass to the limit in (\ref{eq1.32}) for $k \to \infty$, obtaining
$$
 \begin{array}{c}
 \int_{Q_T}  u(t,x) [\overline{D_tv(t,x) + A^*v(t,x)}] dt dx = \\ \\
= - (D_tu - Au, v)-  \int_\Omega  u(0,x) \overline{v(0,x)}] dx
+  \int_{\Sigma_T}Êg(t,x') \overline{Tv(t,x')} dt d\sigma.
\end{array}
$$
We can conclude that $u = S_T(D_tu- Au, u(0,\cdot), g)$.

(II) We prove only the completeness. Let $(u_k)_{k \in \N}$ be a Cauchy sequence in $X_T$.
Then there exists $u$ in $H^{3/4,3/2}(Q_T)$ such that
$$
\|u_k - u\|_{H^{3/4,3/2}(Q_T)} \to 0, \quad (k \to \infty).
$$
It follows that
$$
\begin{array}{c}
\|u_{k|\Sigma_T} - u_{|\Sigma_T}\|_{L^p(0, T; L^2(\Gamma))} +  \|D_tu_k - D_tu\|_{H^{-1/4}(0, T; L^2(\Omega))} \\ \\
 + \|u_k(0,\cdot) - u(0, \cdot)\|_{H^{1/2}(\Omega)} \to 0 \quad (k \to \infty),
\end{array}
$$
and $Au_k \to Au$ in $\mD'(Q_T)$. As $(Au_k)_{k \in \N}$ is a Cauchy sequence in $H^{1/4,1/2}(Q_T)'$, we deduce that
$Au \in H^{1/4,1/2}(Q_T)'$ and
$$
\|Au_k - Au\|_{H^{1/4,1/2}(Q_T)'} \to 0 \quad (k \to \infty).
$$
Let $u_k = S_T(f_k,u^k_0, g_k)$. Then by Proposition \ref{pr1.1} (II) we infer $f_k = D_tu_k - Au_k \to D_tu - Au \in H^{1/4,1/2}(Q_T)'$
and $u_0^k = u_k(0,\cdot)$. Moreover, there exists $g \in L^2(\Sigma_T)$ such that
$$\|Bu_k - g\|_{L^2(\Sigma_T)} =  \|g_k - g\|_{L^2(\Sigma_T)}\to 0 \quad  (k \to \infty). $$
 Let
$v \in H^{1,2}(Q_T)$, with $Cv = 0$ and $v(T,\cdot) = 0$. Then, $\forall \, k \in \N$,
$$
\begin{array}{c}
\int_{Q_T}  u_k(t,x) [\overline{D_tv(t,x) + A^*v(t,x)}] dt dx = \\ \\
 - (f_k, v) - \int_\Omega u_0^k(0,x) \overline{v(0,x)} dx + \int_{\Sigma_T}Êg_k(t,x') \overline{Tv(t,x')} dt d\sigma.
 \end{array}
 $$
 Passing to the limit, for $k \to \infty$, we obtain
 $$
\begin{array}{c}
\int_{Q_T}  u(t,x) [\overline{D_tv(t,x) + A^*v(t,x)}] dt dx = \\ \\
 - (D_tu - Au, v) - \int_\Omega u(0,x) \overline{v(0,x)} dx + \int_{\Sigma_T}Êg(t,x') \overline{Tv(t,x')} dt d\sigma,
 \end{array}
 $$
 which implies  $u = S(D_tu - Au, u(0,\cdot), g)$ and
 $$
 \begin{array}{c}
 \|u - u_k\|_{X_T}
 = \|u - u_k\|_{H^{3/4,3/2}(Q_T)} + \|u_{k|\Sigma_T} - u_{|\Sigma_T}\|_{L^p(0, T; L^2(\Gamma))} \\ \\
 + \|Au - Au_k\|_{H^{1/4,1/2}(Q_T)'} + \|g - Bu_k\|_{L_2(\Sigma_T)} \to 0, \quad (k \to \infty).
 \end{array}
 $$
 \end{proof}

\begin{remark}
{\rm We have already observed that $H^{1/4,1/2}(Q_{T})'$ is a space of distributions in $Q_T$. Let $0 < T < T'$ and $f \in  H^{1/4,1/2}(Q_{T'})'$.
We consider $f_{|(0, T)}$ and $f_{|(T, T')}$ defined in (\ref{eq1.6}) and (\ref{eq1.11A}). They can be identified with the restrictions of $f$
(in the sense of distributions) to  $Q_T$ and $(T, T') \times \Omega$, respectively.
We recall again the notation (\ref{eq1.12}).
Then $f_{|(0, T)} \in H^{1/4,1/2}(Q_{T})'$ and $f_{|(T, T')}(T + \cdot) \in H^{1/4,1/2}(Q_{T'-T})'$. }
\end{remark}

\begin{proposition}\label{pr1.2}
 Assume (H1)-(H2), $0 < T < T'$, $f \in  H^{1/4,1/2}(Q_{T'})'$, $u_0 \in H^{1/2}(\Omega)$ and $g \in L^2(\Sigma_{T'})$. We set
 $$
 \begin{array}{ccc}
 f_0:= f_{|(0, T)}, & g_0:= g_{|\Sigma_T}, &  v_0:= S_T(f_0, u_0, g_0),
 \end{array}
 $$
 $$
 \begin{array}{cccc}
 f_1:= f_{|(T, T')}(T+\cdot), & u_1:= v_0(T, \cdot), &  g_1:= g_{|(T, T') \times \Gamma}(T + \cdot), & v_1:= S_{T'-T}(f_1, u_1, g_1).
 \end{array}
 $$
 Then
 $$
 S_{T'}(f,u_0,g)(t, \cdot) = \left\{\begin{array}{lll}
 v_0(t, \cdot) & \mbox{Êif }Ê& t \in [0, T], \\ \\
 v_1(t - T, \cdot) & \mbox{Êif }Ê& t \in [T, T'].
 \end{array}
 \right.
 $$
 \end{proposition}

\begin{proof} The statement holds if $f \in L^2(Q_{T'})$, $u_0 \in H^1(\Omega)$, $g \in H^{1/4,1/2} (\Sigma_{T'})$.
It can be extended to the general case using an argument of continuity.
\end{proof}

 \begin{remark}\label{re3.21}
 {\rm Proposition \ref{pr1.2} and Lemma \ref{le3.14} imply that, if $T, T' \in \R^+$, $v \in X_T$, $w \in X_{T'}$ and $v(T,\cdot) = w(0, \cdot)$, setting
 $$
 z(t,\cdot):= \left\{\begin{array}{lll}
 v(t, \cdot) & \mbox{Êif }Ê& t \in [0, T], \\ \\
 w(t-T,\cdot) & \mbox{Êif }Ê& t \in [T, T+T'],
 \end{array}
 \right.
 $$
 then $z \in X_{T+T'}$.
 }
 \end{remark}

  \begin{proposition}Ê\label{pr1.3}
  Assume (H1)-(H2) and let $T'>0$. Then there exists $C(T')>0$, such that, $\forall \, T \in (0, T']$, $\forall \, f \in H^{1/4,1/2}(Q_{T})'$,
  $\forall \, u_0 \in H^{1/2}(\Omega)$, $\forall \, g \in L^2(\Sigma_{T})$, it holds
  $$
  \|S_T(f,u_0, g)\|_{X_T} \leq C(T')\left(\|f\|_{H^{1/4,1/2}(Q_{T})'} + \|u_0\|_{H^{1/2}(\Omega)} + \|g\|_{L^2(\Sigma_{T})}\right).
  $$
  \end{proposition}

\begin{proof} Let $f \in H^{1/4,1/2}(Q_{T})'$, $u_0 \in H^{1/2}(\Omega)$, $g \in L^2(\Sigma_{T})$. We consider $\widetilde f \in H^{1/4,1/2}(Q_{T'})'$
defined in (\ref{eq1.7A}) (with $z$ replacing $f$) and the trivial extension $\widetilde g$ of $g$ to $\Sigma_{T'}$.
Observe that, by Proposition \ref{pr1.2}, $S_T(f,u_0,g)$ coincides with the restriction of $S_{T'}(\tilde f, u_0, \tilde g)$ to $Q_T$, so that
$$AS_T(f,u_0,g) = AS_{T'}(\tilde f, u_0, \tilde g)_{|(0, T)}.$$
Hence, by Lemma \ref{le1.3} (I)-(II), we deduce
 $$
 \begin{array}{c}
 \|S_T(f,u_0, g)\|_{X_T} = \|S_T(f,u_0, g)\|_{H^{3/4,3/2}(Q_T)} + \|S_T(f,u_0, g)_{|\Sigma_T}\|_{L^{p}(0, T; L^2(\Gamma))} \\ \\
 +  \|A S_T(f,u_0, g)\|_{H^{1/4,1/2}(Q_T)'} + \|g\|_{L^2(\Sigma_T)} \\ \\
 \leq \|S_{T'}(\widetilde f,u_0, \widetilde g)\|_{H^{3/4,3/2}(Q_{T'})} +  \|S_{T'}(\widetilde f,u_0, \widetilde g)_{|\Sigma_{T'}}\|_{L^{p}(0, T'; L^2(\Gamma))} \\ \\
 +  C_1  \|A S_{T'}(\widetilde f,u_0, \widetilde g)\|_{H^{1/4,1/2}(Q_{T'})'} + \|\widetilde g\|_{L^2(\Sigma_{T'})} \\ \\
 \leq C_1(T') \left (\|\widetilde f\|_{H^{1/4,1/2}(Q_{T'})'} + \|u_0\|_{H^{1/2}(\Omega)} + \|\widetilde g\|_{L^2(\Sigma_{T'})} \right) \\ \\
 \leq C_1(T') \left(\|f\|_{H^{1/4,1/2}(Q_{T})'}^2 + \|u_0\|_{H^{1/2}(\Omega)} + \||g\|_{L^2(\Sigma_{T})} \right).
\end{array}
 $$
\end{proof}

 \begin{proposition}Ê\label{pr1.4}
  Assume (H1)-(H2) and let $T' \in \R^+$. Then there exists $C(T') >0$ such that, $\forall \, T \in (0, T']$, $\forall \, f \in L^2(Q_{T})$,
  $\forall \, u_0 \in H^{1}(\Omega)$, $\forall \, g \in H^{1/4,1/2}(\Sigma_{T})$, it holds
  $$
  \|S_T(f,u_0, g)\|_{H^{1,2}(Q_T)} \leq C(T')(\|f\|_{L^2(Q_{T})} + \|u_0\|_{H^{1}(\Omega)} + \|g\|_{H^{1/4,1/2}(\Sigma_T)}.
  $$
  \end{proposition}

 \begin{proof} We consider $\widetilde f \in L^2(Q_{T'})$, the extension of $f$ to $Q_{T'}$,
 and $\widetilde g$, the extension of $g$ to $\Sigma_{T'}$. Observe that, by Proposition \ref{pr1.2},
 $S_T(f,u_0,g)$ coincides with the restriction of $S_{T'}(\widetilde f, u_0, \widetilde g)$ to $Q_T$.
 So, by Lemma \ref{le1.1} (II), we have
 $$
 \begin{array}{c}
 \|S_T(f,u_0, g)\|_{H^{1,2}(Q_T)}   \leq \|S_{T'}(\widetilde f,u_0, \widetilde g)\|_{H^{1,2}(Q_{T'})}  \\ \\
 \leq C_1(T') \left(\|\widetilde f\|_{L^{2}(Q_{T'})} + \|u_0\|_{H^{1}(\Omega)} + \|\widetilde g\|_{H^{1/4,1/2}(\Sigma_{T'})}\right) \\ \\
 \leq C_1(T') \left(\|f\|_{L^{2}(Q_{T})}  + \|u_0\|_{H^{1/2}(\Omega)} + C\|g\|_{H^{1/4,1/2}(\Sigma_{T})} \right) \\ \\
  \leq C_2(T') \left(\|f\|_{L^{2}(Q_{T})}  + \|u_0\|_{H^{1/2}(\Omega)} + \|g\|_{H^{1/4,1/2}(\Sigma_{T})} \right).
\end{array}
 $$
\end{proof}

 \section{Convolution terms}\label{seA4}

 \setcounter{equation}{0}

In this section, we examine the convolution terms appearing in system (\ref{eqA2.15}).
Let us consider the following abstract situation. Given $h \in L^1(0, T)$, $f \in L^2(0, T; V)$,
with $V$ Hilbert space, we set
\begin{equation}
(h \ast f)(t) := \int_0^t h(t-s) f(s) ds.
\end{equation}
By Young's inequality, $h \ast f \in L^2(0, T; V)$ and
$$
\|h\ast f\|_{L^2(0, T; V)} \leq \|h\|_{L^1(0, T)}Ê\|f\|_{L^2(0, T; V)}.
$$
The first result says

\begin{lemma}\label{le1.6}
Let $V$ be a Hilbert space, $\beta \in (0, 1/2)$, $h \in L^1(0, T)$ and $f \in H^\beta(0, T; V)$. Then $h \ast f \in H^\beta(0, T; V)$.
Moreover, there exists $C >0$, independent of $T$, $h$, $f$, such that
$$
\|h \ast f\|_{H^\beta(0, T; V)} \leq C \|h\|_{ L^1(0, T)} \| f\|_{H^\beta(0, T; V)}.
$$
\end{lemma}

\begin{proof} By Cauchy-Schwarz inequality, we get
$$
\|(h \ast f)(t)\|_V ^2 \leq \|h\|_{L^1(0, T)} \int_0^t |h(s)| \|f(t-s)\|_V ^2 ds, \quad \forall \, t \in (0,T).
$$
Then it follows
$$
\begin{array}{c}
\int_0^T t^{-2\beta} \|(h \ast f)(t)\|_V ^2 dt \leq \|h\|_{L^1(0, T)} \int_0^T t^{-2\beta} (\int_0^t |h(s)| \|f(t-s)\|_V ^2 ds) dt \\ \\
=  \|h\|_{L^1(0, T)} \int_0^T |h(s)| (\int_s^T t^{-2\beta} \|f(t-s)\|_V ^2 dt) ds \\ \\
\leq  \|h\|_{L^1(0, T)}^2 \int_0^T t^{-2\beta}  \|f(t)\|_V ^2 dt.
\end{array}
$$
Moreover, if $0 < s < t < T$,
$$
\begin{array}{c}
\|(h \ast f)(t) - ( h \ast f)(s)\|_V \leq \int_0^s |h(\tau)| \|f(t-\tau) - f(s-\tau)\|_V d\tau + \int_s^t |h(\tau)|  \|f(t-\tau)\|_V d\tau \\ \\
\leq \|h\|_{L^1(0, T)}^{1/2}Ê[(\int_0^s |h(\tau)| \|f(t-\tau) - f(s-\tau)\|_V^2 d\tau)^{1/2} + ÊÊ(\int_s^t |h(\tau)|  \|f(t-\tau)\|_V^2 d\tau)^{1/2}ÊÊÊÊ],
\end{array}
$$
so that
$$
\begin{array}{c}
\int_0^T (\int_0^t \frac{\|(h \ast f)(t) - ( h \ast f)(s)\|_V^2Ê}{(t-s)^{1+2\beta}} ds ) dt \\ \\
\leq 2 \|h\|_{L^1(0, T)} [Ê\int_0^T (\int_0^t (Ê\int_0^s |h(\tau)| \|f(t-\tau) - f(s-\tau)\|_V^2 d\tau) (t-s)^{-1-2\beta} ds ) dt \\ \\
+ \int_0^T (\int_0^t (\int_s^t |h(\tau)|  \|f(t-\tau)\|_V^2 d\tau) (t-s)^{-1-2\beta} ds ) dt ].
\end{array}
$$
We have
$$
\begin{array}{c}
\int_0^T (\int_0^t (Ê\int_0^s |h(\tau)| \|f(t-\tau) - f(s-\tau)\|_V^2 d\tau) (t-s)^{-1-2\beta} ds ) dt \\ \\
= \int_0^T  |h(\tau)| (\int_\tau^T (\int_\tau^t (t-s)^{-1-2\beta}  \|f(t-\tau) - f(s-\tau)\|_V^2 ds ) dt) d\tau \\ \\
\leq \|h\|_{L^1(0, T)}Ê\int_0^T (\int_0^t \frac{\|f(t) - f(s)\|_VÊÊ}{(t-s)^{1+2\beta}Ê} ds) dt.
\end{array}
$$
Finally, we deduce
$$
\begin{array}{c}
\int_0^T (\int_0^t (\int_s^t |h(\tau)|  \|f(t-\tau)\|_V^2 d\tau) (t-s)^{-1-2\beta} ds ) dt \\ \\
= \int_0^T (\int_0^t  |h(\tau)| \|f(t-\tau)\|_V^2 (\int_0^\tau (t-s)^{-1-2\beta} ds ) d\tau ) dt \\ \\
\leq \frac{1}{2\beta}Ê\int_0^T (\int_0^t  |h(\tau)| (t-\tau)^{-2\beta}Ê \|f(t-\tau)\|_V^2 d\tau ) dt \\ \\
\leq \frac{1}{2\beta} \|h\|_{L^1(0, T)} \int_0^T  t^{-2\beta}Ê \|f(t)\|_V^2 dt.
 \end{array}
$$
The conclusion follows.
\end{proof}

Now, having in mind the case $h \in L^1(0, T)$ and $f \in H^{1/4,1/2}(Q_T)'$, we want to define the convolution
$h \ast z$ when $h \in L^1(0, T)$ and $z \in Y_T'$ (cf. definition \ref{alfa}).

\begin{lemma}\label{le1.7}
There exists a unique bilinear and continuous mapping $(h,z) \to h \ast z$ from $L^1(0, T) \times Y_T'$ to $Y_T'$
extending the convolution mapping $(h,g) \to h \ast g$ from  $L^1(0, T) \times L^2(0, T; V_2)$ to $L^2(0, T; V_2)$. Moreover,
we can find a positive constant $C$, independent of $T$, $h$, $z$, such that
$$
\|h \ast z\|_{Y_T'} \leq C \|h\|_{L^1(0, T)} \|z\|_{Y_T'}.
$$
\end{lemma}
\begin{proof} The uniqueness of the extension of the convolution follows from Lemma \ref{le1.2}.  If $h \in L^1(0, T)$, $g \in L^2(0, T; V_2)$
and $f \in Y_T$, we have
$$
\begin{array}{c}
(h \ast g, f) =  \int_0^T ((h \ast g)(t), f(t))_{V_2} dt = \int_0^T h(s) (\int_0^{T-s} (g(\tau), f(\tau +s))_{V_2} d\tau) ds \\ \\
= \int_0^T h(s) (\int_0^{T} (g(\tau), f_s(\tau))_{V_2} d\tau) ds = \int_0^T h(s) (g, f_s) ds,
\end{array}
$$
with $f_s$ as in (\ref{eq1.4}). So, if $z \in Y_T'$ and $f \in Y_T$, we define
\begin{equation}\label{eq1.8}
(h \ast z, f):= \int_0^T h(s) (z, f_s) ds.
\end{equation}
Observe that (\ref{eq1.8}) is well defined, because, in force of Lemma \ref{le1.1} (III), the mapping $s \to (z, f_s)$ is continuous and bounded.
We have also
$$
|(h \ast z, f)| \leq \int_0^T |h(s)| |(z, f_s)| ds \leq C \|h\|_{L^1(0, T)}Ê\|z\|_{Y_T'}Ê\|f\|_{Y_T},
$$
with $C$ independent of $T$, $h$, $z$, $f$.
The conclusion follows.
\end{proof}

\begin{lemma} \label{le1.8}
Let $0 < T < T'$, $h \in L^1(0, T')$, $z \in Y_{T'}'$. We set $z_0:= z_{|(0, T)}$, $z_1:= z_{|(T, T')}$, $h_0:= h_{||0, T)}$, $h_1:= h_{|(T, T')}$.
Then, recalling the notation (\ref{eq1.12}), there hold

(I) \quad $(h\ast z)_{|(0, T)} = h_0 \ast z_0$;
 \vskip0.1truecm
(II)\quad
$(h\ast z)_{|(T, T')}(\cdot + T) = (h \ast {\widetilde z}_0)_{|(T, T')}(\cdot + T) + h_{|(0, T'-T)} \ast [z_1(\cdot + T)] $
 \vskip0.1truecm
$\quad \quad \quad = ({\widetilde h}_0 \ast z)_{|(T, T')} (\cdot + T) + h_1(\cdot + T) \ast z_{|(0, T'-T)};$
 \vskip0.1truecm

(III) if $T' \leq 2T$
 \vskip0.1truecm
$\quad \quad \quad
(h\ast z)_{|(T, T')}(\cdot + T)
= ({\widetilde h}_0 \ast {\widetilde z}_0 )_{|(T, T')}(\cdot + T) + h_1(\cdot + T) \ast z_{0|(0, T-T')}Ê+ h_{0|(0, T'-T)} \ast [z_1(\cdot + T)].
$
\end{lemma}

\begin{proof} By Lemma \ref{le1.2}, it suffices to consider $z \in L^2(0, T'; V_2)$. In this case, (I) is obvious.
Concerning (II), we have, for $t \in (0, T' - T)$,
$$
\begin{array}{c}
(h \ast z)(T+t) = \int_0^{T+t}Êh(T+t-s)  z(s) ds \\ \\
= \int_0^{T+t} h(T+t-s)  {\widetilde z}_0(s) ds + \int_0^{t} h(t-s)  z_1(s+T) ds
\end{array}
$$
and the first identity in (II) is proved. The second identity follows inverting the roles of $h$ and $z$, as
$$
(h \ast z)(T+t) = \int_0^{T+t}Êh(s)  z(T+t-s) ds.
$$
We have also
$$
\int_0^{T+t} h(T+t-s){\widetilde z}_0(s) ds = \int_0^{T+t}{\widetilde h}_0(T+t-s){\widetilde z}_0(s) ds + \int_0^t h_1(T+t-s){\widetilde z}_0(s)ds,
$$
which implies (III) and completes the proof.
\end{proof}

\begin{remark}\label{re4.4}
{Ê\rm If we think of the final formula in (III) as a function of $(z_1,h_1)$, we see that it is affine, in spite of the fact that the convolution,
as a function of $(h,z)$,  does not enjoy this property. This will be crucial to prove global existence for solution to Problem 3 (see Section 5).}
\end{remark}

\section{The memory term }\label{seA5}

\setcounter{equation}{0}

In this section we study some properties of the term $\Psi(v)$ defined in (\ref{eqA2.12}).
We recall that the (usually nonlinear) operator $\mW$ fulfills the conditions (C1)-(C3).
Let $v \in H^{3/4,3/2}(Q_T)$. We have
$$
\mM(u_0 + 1 \ast v)(t) = \mM u_0 + (1 \ast \mM v)(t), \quad t \in (0, T).
$$
On account of (H8), observe that the function $\int_\Omega \omega_1(x) v(\cdot,x) dx \in H^{3/4}(0, T)$.
Moreover, as $v_{|\Sigma_T} \in H^{1/2,1}(\Sigma_T)$, then $\int_\Omega \omega_2(y) v(\cdot,y) d\sigma \in H^{1/2}(0, T)$.
So,  $\mM(u_0 + 1 \ast v) \in H^{3/2}(0, T) \hookrightarrow C([0, T]) \cap BV([0, T])$ and $\Psi(v)$ is well defined.
More in general, taking  $v \in H^{3/4,3/2}(Q_\tau)$ with $0 < \tau \leq T$, we define
\begin{equation}
\Psi(v)(t,y):= D_t[\mF (\mW (\mM(u_0(x) + 1 \ast v(t,x)))](t,y), \quad (t,y) \in \Sigma_\tau.
\end{equation}
We observe that, if $v_1,v_2 \in H^{3/4,3/2}(Q_\tau)$ and $v_{1|Q_{\tau_1}} = v_{2|Q_{\tau_1}}$  for some $\tau_1 \in [0, \tau]$,
then
$$
\Psi(v_1)(t,y) = \Psi(v_2)(t,y) \quad \mbox{ in }Ê\Sigma_{\tau_1}.
$$
Further, recalling  definition (\ref{eqA1.11}), we have
\begin{equation}
\Psi(v)(t,y) = D_t(E_1 \ast r_v)(t) u_A(t,y) + (E_1 \ast r_v)(t) D_tu_A(t,y) + D_tE_0(t,y), \quad (t,y) \in \Sigma_\tau,
\end{equation}
with
\begin{equation}
r_v(t):= \mW (\mM u_0 + 1 \ast \mM v)(t).
\end{equation}
The following result holds
\begin{lemma}\label{le2.1}
Let $v_1, v_2 \in H^{3/4,3/2}(Q_\tau)$, with $0 < \tau \leq T$. Then
$$
\|\Psi(v_1) - \Psi(v_2)\|_{L^2(\Sigma_\tau)} \leq C(T) \tau^{1/2} \|v_1 - v_2\|_{H^{3/4,3/2} (Q_\tau)}.
$$
\end{lemma}
\begin{proof} First, we estimate $\| D_t[E_1 \ast (r_{v_1} - r_{v_2})]\|_{C([0, \tau])}$. Employing Lemma \ref{le1.1A}, we have
$$
\begin{array}{c}
\|D_t[E_1 \ast (r_{v_1} - r_{v_2})]\|_{C([0, \tau])} \leq C_1(T) \|r_{v_1} - r_{v_2}\|_{C([0, \tau])} \\ \\
\leq C_1(T) L \|1 \ast  \mM (v_1 - v_2)\|_{C([0, \tau])} \leq C_1(T) L \tau^{1/2} \| \mM (v_1 - v_2)\|_{L^2(0, \tau)} \\ \\
\leq  C_1(T) L \tau^{1/2} \left(\|\omega_1\|_{L^2(\Omega)} \|v_1 - v_2\|_{L^2(Q_\tau)}
+ \|\omega_2\|_{L^2(\Gamma)}Ê\|(v_1 - v_2)_{|\Sigma_\tau}\|_{L^2(\Sigma_\tau)}\right) \\ \\
\leq C_2(T) L \tau^{1/2} \left(\|\omega_1\|_{L^2(\Omega)} \tau^{1/2} \|v_1 - v_2\|_{H^{3/4}((0, \tau); L^2(\Omega))}
+ \|\omega_2\|_{L^2(\Gamma)}Ê\|v_1 - v_2\|_{L^2((0, \tau); H^{3/2}(\Omega))}\right) \\ \\
\leq C_3(T) \tau^{1/2} \|v_1 - v_2\|_{H^{3/4,3/2} (Q_\tau)}.
\end{array}
$$
It follows
$$
\begin{array}{c}
\|E_1 \ast (r_{v_1} - r_{v_2})]\|_{C([0, \tau])} = \|1 \ast D_t[E_1 \ast (r_{v_1} - r_{v_2})]\|_{C([0, \tau])} \\ \\
\leq \tau \|D_t[E_1 \ast (r_{v_1} - r_{v_2})]\|_{C([0, \tau])} \leq C_3(T) \tau^{3/2} \|v_1 - v_2\|_{H^{3/4,3/2} (Q_\tau)}.
\end{array}
$$
So we have
$$
\begin{array}{c}
\|\Psi(v_1) - \Psi(v_2)\|_{L^2(\Sigma_\tau)}  \leq
\|[ÊD_t(E_1 \ast (r_{v_1} - r_{v_2}))](t)  u_A(t,y)\| _{L^2(\Sigma_\tau)}  \\ \\
+ \|[E_1 \ast (r_{v_1} - r_{v_2})] (t) D_tu_A(t,y) \| _{L^2(\Sigma_\tau)} \\ \\
\leq \|[ÊD_t(E_1 \ast (r_{v_1} - r_{v_2}))]\|_{C([0, \tau])}Ê\|u_A\| _{L^2(\Sigma_\tau)}
+  \|[E_1 \ast (r_{v_1} - r_{v_2})]\|_{C([0, \tau])} \|D_t u_A\| _{L^2(\Sigma_\tau)} \\ \\
\leq C_4(T) \tau^{1/2} \|v_1 - v_2\|_{H^{3/4,3/2} (Q_\tau)}.
\end{array}
$$
\end{proof}
Let us consider now $\delta >0$ such that $0 < \tau < \tau + \delta \leq T$. We fix $w \in X_\tau$ (recall definition (\ref{de3.1})).
If $v \in  X_\delta$ and $v(0, \cdot) = w(\tau, \cdot)$, we set
\begin{equation}\label{eq2.7}
V(t,\cdot) = \left\{\begin{array}{ll}
w(t, \cdot) & \mbox{ if } t \in [0, \tau], \\ \\
v(t-\tau, \cdot) & \mbox{ if } t \in [\tau, \tau + \delta]
\end{array}
\right.
\end{equation}
and define
\begin{equation}
\Psi(w,v)(t,y):= \Psi(V)(\tau+t,y), \quad t \in (0, \delta), \, y \in  \Gamma.
\end{equation}
On account of Remark \ref{re3.21}, we observe that  $V \in X_{\tau +\delta}$. Moreover it holds

\begin{corollary}\label{co2.1}
Let $v_1, v_2 \in X_\delta$ be such that $v_1(0, \cdot) = v_2(0,\cdot) = w(\tau, \cdot)$. Then
$$
\|\Psi(w,v_1) - \Psi(w,v_2)\|_{L^2(\Sigma_\delta)}Ê\leq C\delta^{1/2}Ê\|v_1 - v_2\|_{H^{3/4,3/2}(Q_\delta)},
$$
with $C$ independent of $\tau$ in $[0, T)$, $\delta \in (0, T-\tau]$, $w, v_1, v_2$.
\end{corollary}

\begin{proof} Indicating by $V_j$ the function obtained replacing $v$ with $v_j$ ($j \in \{1,2\}$) in (\ref{eq2.7}),
we have
$$
\mM(u_0 + 1\ast V_j)(t) = \left\{\begin{array}{ll}
\mM u_0 + 1 \ast \mM(w)(t) & \mbox{ if  } t \in [0, \tau], \\ \\
\mM u_0 + 1 \ast \mM(w)(\tau) + 1 \ast \mM(v_j)(t-\tau) & \mbox{ if  } t \in [\tau, \tau + \delta].
\end{array}
\right.
$$
So
$$
\|\mW(\mM(u_0 + 1\ast V_1)) - \mW(\mM(u_0 + 1\ast V_1))\|_{C([0, \tau+\delta])} \leq L \|1 \ast \mM(v_1- v_2)\|_{C([0, \delta])},
$$
and the conclusion follows as in the proof of Lemma \ref{le2.1}.
\end{proof}

\section{Weak solutions to Problem \ref{P3}}\label{se6}

\setcounter{equation}{0}

In this section we begin to study the inverse problem reformulated as Problem \ref{P3}. Here we shall limit ourselves to consider
weak solutions, in the sense that we shall not search for a solution $(v,h)\in H^{1,2}(Q_T) \times L^2(0, T)$, but rather
a solution $(v,h)\in X_T \times L^2(0, T)$,
with $X_T$ defined as in  \ref{de3.1}. So, we are going to consider system (\ref{eqA2.6}), with the following (generalized) assumptions:

\medskip

{\it (K1) (H1)-(H2) are fulfilled;

\medskip

(K2) $v^* \in H^{1/4,1/2}(Q_T)'$;

\medskip

(K3) $\psi_1, z_0 \in L^2(\Omega)$;

\medskip

(K4) $v_0 \in H^{1/2}(\Omega)$;

\medskip

(K5) $z_1 \in L^2(\Gamma)$, $v^*_\Gamma \in L^2(\Sigma_T)$;

\medskip

(K6) $h^* \in L^2((0, T))$.

}

\medskip

We introduce the following auxiliary function $V_0$ solution in $X_T$ of
\begin{equation}\label{eq3.5}
\left\{\begin{array}{l}
D_tV_0(t,x)   = AV_0(t,x)  + v^*(t,x),  \quad  (t,x) \in Q_T,  \\ \\
V_0(0,x) = v_0(x),  \quad  x \in  \Omega \\ \\
BV_0(t,y)  = \Psi(0)(t,y) + v^*_\Gamma (t,y), \quad (t,y) \in \Sigma_\tau.
\end{array}
\right.
\end{equation}

\begin{remark}
{\rm Of course, $V_0$ is the solution of (\ref{eq3.5}) in the sense of Theorem \ref{th1.4}.
Moreover, we shall consider the restrictions of $v^*$ to $(0, \tau)$, with $0 < \tau \leq T$ (see (\ref{eq1.6}))
usually writing  $v^*$ instead of
 $v^*_{|(0, \tau)}$. We shall follow the same convention for  other restrictions, if this is not likely to produce confusion.}
\end{remark}

We begin with a result of local existence.

\begin{lemma}\label{le3.1}
Assume (K1)-(K6). Then, $\forall \, r >0$, there exists $\tau(r) \in (0, T]$, such that, if $0 < \tau \leq \tau(r)$,
Problem (\ref{eqA2.6}) admits a unique solution $(v,h) \in X_\tau \times L^2(0, \tau)$ and satisfying
$$
\|v-V_0\|_{X_\tau} + \|h - h^*\|_{L^2(0, \tau)} \leq r.
$$
\end{lemma}

\begin{proof} We start by observing that $Z_\tau:= X_\tau \times L^2(0, \tau)$ is a Banach space with the norm
\begin{equation}
\|(z,h)\|_{Z_\tau} := \|z\|_{X_\tau} + \|h\|_{L^2(0, \tau)}.
\end{equation}
We set
\begin{equation}
Z_\tau(r):= \{(v,h) \in Z_\tau :  \|v - V_0\|_{X_\tau} + \|h - h^*\|_{L^2(0, \tau)} \leq r \},
\end{equation}
which is a closed subset of $Z_\tau$. Let $(V,H) \in Z_\tau$. We consider the element $(v,h)$ in $Z_\tau$, such that
$$
\left\{\begin{array}{l}
D_tv(t,x)   = Av(t,x) + H \ast AV(t,x) + v^*(t,x) \\ \\
 - [(\psi_1,V(t,\cdot)) + H \ast (\psi_1, V(t,\cdot))]Êz_0(x), \quad  (t,x) \in Q_\tau,  \\ \\
v(0,x) = v_0(x), \quad x \in \Omega \\ \\
Bv(t,y)  = - V(t,y) + \Psi(V)(t,y) - H \ast BV(t,y) +    [(\psi_1,V(t,\cdot)) + H \ast (\psi_1, V(t,\cdot))] z_1(y)Ê\\ \\
+ v^*_\Gamma (t,y), \quad (t,y) \in \Sigma_\tau \\ \\
h(t) = h^*(t) - (\psi_1,V(t,\cdot)) - H \ast (\psi_1, V(t,\cdot))(t), \quad t \in (0, \tau),
\end{array}
\right.
$$
and the map $P(V,H):=(v,h)$. It is clear that $(v,h)$ is a solution of $(\ref{eqA2.6})$ if and only if it is a fixed point of $P$.
Let $r >0$. We show that, if $\tau$ is sufficiently small, $P$ maps $Z_\tau(r)$ into itself. Let $(V,H) \in Z_\tau(r)$. Then $(v-V_0, h - h^*)$ satisfies
$$
\left\{\begin{array}{l}
D_t(v - V_0)(t,x)   = A(v - V_0)(t,x) + H \ast AV(t,x) \\ \\
 - [(\psi_1,V(t,\cdot)) + H \ast (\psi_1, V(t,\cdot))]Êz_0(x), \quad  (t,x) \in Q_\tau,  \\ \\
(v-V_0)(0,x) = 0, \quad x \in \Omega \\ \\
B(v - V_0)(t,y)  = - V(t,y) + \Psi(V)(t,y) - \Psi(0)(t,y)  - H \ast BV(t,y)  \\ \\  +  [(\psi_1,V(t,\cdot)) + H \ast (\psi_1, V(t,\cdot))] z_1(y),
\quad (t,y) \in \Sigma_\tau \\ \\
h(t) - h^*(t)  =  - (\psi_1,V(t,\cdot)) - H \ast (\psi_1, V(t,\cdot))(t), \quad t \in (0, \tau).
\end{array}
\right.
$$
In the following part of the proof we shall indicate by $C, C_1, C_2, \dots$ some positive constants independent of $\tau$, $r$, $V$ and $H$.
By Lemma \ref{le1.1A} we have
\begin{equation}\label{eq3.8}
\begin{array}{c}
\|(\psi_1,V(t,\cdot)) + H \ast (\psi_1, V(t,\cdot))\|_{L^2(0, \tau)}Ê
\leq \|\psi_1\|_{L^2(\Omega)} \|V\|_{L^2(Q_\tau)} \left(1 + \|H\|_{L^1(0, \tau)}\right) \\ \\
\leq C_1 \tau^{1/2} \|V\|_{H^{3/4}(0, \tau; L^2(\Omega))} \left(1 + \tau^{1/2}\|H\|_{L^2(0, \tau)}\right) \\ \\
\leq C_1 \tau^{1/2} \left(r + \|V_0\|_{H^{3/4}(0, T; L^2(\Omega))}\right) \left[1 + \tau^{1/2} (r + \|h^*\|_{L^2(0, T)})\right].
\end{array}
\end{equation}
Now we apply Proposition \ref{pr1.3} in order to estimate $\|v - V_0\|_{X_\tau}$. We get
\begin{equation}
\begin{array}{c}
\|v - V_0\|_{X_\tau} \leq C_1\big (\| H \ast AV\|_{H^{1/4,1/2}(Q_\tau)'}Ê+ \|[(\psi_1,V(t,\cdot))
+ H \ast (\psi_1, V(t,\cdot))]Êz_0(x)\|_{H^{1/4,1/2}(Q_\tau)'} \\ \\
+ \|V_{|\Sigma_\tau}\|_{L^2(\Sigma_\tau)} + \|\Psi(V) - \Psi(0)\|_{L^2(\Sigma_\tau)}  + \|H \ast BV\|_{L^2(\Sigma_\tau)} \\ \\ + \|[(\psi_1,V(t,\cdot))
+ H \ast (\psi_1, V(t,\cdot))] z_1(y)\|_{L^2(\Sigma_\tau)} \big).
\end{array}
\end{equation}
By Lemmas \ref{le1.7} and \ref{le1.3} (I), we obtain
\begin{equation}\label{eq3.10}
\begin{array}{c}
\| H \ast AV\|_{H^{1/4,1/2}(Q_\tau)'}Ê\leq C_1 \|H\|_{L^1(0, \tau)} \|AV\|_{H^{1/4,1/2}(Q_\tau)'}Ê
\leq C_1 \tau^{1/2} \|H\|_{L^2(0, \tau)} \|AV\|_{H^{1/4,1/2}(Q_\tau)'} \\ \\
\leq C_1 \tau^{1/2} \left(r + \|h^* \|_{L^2(0, \tau)}\right) \left(r + \|AV_0\|_{H^{1/4,1/2}(Q_\tau)'}\right) \\ \\
\leq C_1 \tau^{1/2} \left(r + \|h^* \|_{L^2(0, \tau)}\right) \left(r + C_2 \|AV_0\|_{H^{1/4,1/2}(Q_T)'}\right).
\end{array}
\end{equation}
From Lemma \ref{le1.3} (III) and (\ref{eq3.8}), it follows
\begin{equation}
\begin{array}{c}
\|[(\psi_1,V(t,\cdot)) + H \ast (\psi_1, V(t,\cdot))]Êz_0(x)\|_{H^{1/4,1/2}(Q_\tau)'} \\ \\
 \leq \tau^{1/4} \|[(\psi_1,V(t,\cdot)) + H \ast (\psi_1, V(t,\cdot))]Êz_0(x)\|_{L^2(Q_\tau)} \\ \\
 = \tau^{1/4} \|[(\psi_1,V(t,\cdot)) + H \ast (\psi_1, V(t,\cdot))]\|_{L^2(0, \tau)}ÊÊ\|z_0(x)\|_{L^2(\Omega)} \\ \\
 \leq  C_2\tau^{3/4} \left(r + \|V_0\|_{H^{3/4}(0, T; L^2(\Omega))}\right) \left[1 + \tau^{1/2} (r + \|h^*\|_{L^2(0, T)})\right].
 \end{array}
\end{equation}
Moreover,
\begin{equation}
\|V_{|\Sigma_\tau}\|_{ L^2(\Sigma_\tau)}Ê\leq \tau^{\frac{1}{2} - \frac{1}{p}} \|V_{|\Sigma_\tau}\|_{ L^p(0, \tau; L^2(\Gamma))}Ê
\leq \tau^{\frac{1}{2} - \frac{1}{p}} \left(r + \|V_{0|\Sigma_T}\|_{ L^p(0, T; L^2(\Gamma))}\right).
\end{equation}
Next, by Lemma \ref{le2.1}, we have
\begin{equation}
\|\Psi(V) - \Psi(0)\|_{L^2(\Sigma_\tau)} \leq C(T) \tau^{1/2} \|V\|_{H^{3/4,3/2}(Q_\tau)} \leq  C(T) \tau^{1/2} \left(r + \|V_0\|_{H^{3/4,3/2}(Q_T)}\right),
\end{equation}
\begin{equation}
\|H \ast BV\|_{L^2(\Sigma_\tau)} \leq \tau^{1/2} \|H\|_{L^2(0, \tau)} \|BV\|_{L^2(\Sigma_\tau)}
\leq \tau^{1/2} \left(r +  \|h^*\|_{L^2(0, T)}) (r + \|BV_0\|_{L^2(\Sigma_T)}\right).
\end{equation}
Finally,
\begin{equation}\label{eq3.15}
\begin{array}{c}
 \|[(\psi_1,V(t,\cdot)) + H \ast (\psi_1, V(t,\cdot))] z_1(y)\|_{L^2(\Sigma_\tau)}\\ \\
 =  \|(\psi_1,V(t,\cdot)) + H \ast (\psi_1, V(t,\cdot))\|_{L^2(0, \tau)}Ê\| z_1\|_{L^2(\Gamma)} \\ \\
 \leq C_1 \tau^{1/2} (r + \|V_0\|_{H^{3/4}(0, T; L^2(\Omega))})\left [1 + \tau^{1/2} \left(r + \|h^*\|_{L^2(0, T)}\right)\right].
 \end{array}
\end{equation}
From (\ref{eq3.8}) and (\ref{eq3.10})-(\ref{eq3.15}), we deduce the estimate
$$
\|v - V_0\|_{X_\tau}  + \|h - h^*\|_{L^2((0, \tau))} \leq C(r) \tau^\epsilon,
$$
for some $\epsilon >0$. Choosing $\tau$ such that $C(r) \tau^\epsilon \leq r$, we obtain that $P: Z_\tau(r) \to Z_\tau(r)$.

Now, for $j \in \{1,2\}$,  we take $(V_j, H_j) \in Z_\tau(r)$ and we put $(v_j, h_j):= P(V_j, H_j)$. Then the pair $(v_1-v_2, h_1-h_2)$ solves the system
$$
\left\{\begin{array}{l}
D_t(v_1 - v_2)(t,x)   = A(v_1 - v_2) (t,x) + H_1 \ast AV_1(t,x) -  H_2 \ast AV_2(t,x) \\ \\
 - [(\psi_1,V_1(t,\cdot) - V_2(t,\cdot) ) + H_1 \ast (\psi_1, V_1(t,\cdot)) - H_2 \ast (\psi_1, V_2(t,\cdot))]Êz_0(x), \quad  (t,x) \in Q_\tau,  \\ \\
v_1(0,x) - v_2(0,x) = 0,  \quad x \in \Omega \\ \\
B(v_1 - v_2) (t,y)  = V_2(t,y) - V_1(t,y) + \Psi(V_1)(t,y) - \Psi(V_2)(t,y)  - H_1 \ast BV_1(t,y) + H_2 \ast BV_2(t,y)  \\ \\
+    [(\psi_1,V_1(t,\cdot) - V_2(t,\cdot)) + H_1 \ast (\psi_1, V_1(t,\cdot)) - H_2 \ast (\psi_1, V_2(t,\cdot))] z_1(y),
\quad (t,y) \in \Sigma_\tau \\ \\
h_1(t) - h_2(t)  =  - (\psi_1,V_1(t,\cdot) - V_2(t,\cdot)) - H_1 \ast (\psi_1, V_1(t,\cdot))(t) + H_2 \ast (\psi_1, V_2(t,\cdot))(t) , \quad t \in (0, \tau),
\end{array}
\right.
$$
so that
$$
\begin{array}{c}
\|(v_1-v_2, h_1- h_2)\|_{Z_\tau} \leq C\big (\| H_1 \ast AV_1 -  H_2 \ast AV_2\|_{H^{1/4,1/2}(Q_\tau)'}Ê\\ \\
+ \|- [(\psi_1,V_1(t,\cdot) - V_2(t,\cdot) ) + H_1 \ast (\psi_1, V_1(t,\cdot)) - H_2 \ast (\psi_1, V_2(t,\cdot))]Êz_0(x)\| _{H^{1/4,1/2}(Q_\tau)'}Ê\\ \\
+\|(V_2 - V_1)_{|\Sigma_\tau}\|_{L^2(\Sigma_\tau)}Ê+ \|\Psi(V_1) - \Psi(V_2)\|_{L^2(\Sigma_\tau)} + \|H_1 \ast BV_1  - H_2 \ast BV_2\|_{L^2(\Sigma_\tau)} \\ \\Ê
+ \|[(\psi_1,V_1(t,\cdot) - V_2(t,\cdot)) + H_1 \ast (\psi_1, V_1(t,\cdot)) - H_2 \ast (\psi_1, V_2(t,\cdot))] z_1(y)\|_{L^2(\Sigma_\tau)} \\ \\
+ \|(\psi_1,V_1(t,\cdot) - V_2(t,\cdot)) + H_1 \ast (\psi_1, V_1(t,\cdot))(t) - H_2 \ast (\psi_1, V_2(t,\cdot))\|_{L^2(0, \tau)}\big ).
\end{array}
$$
By Lemma \ref{le1.1A} we have
\begin{equation}\label{eq3.16}
\begin{array}{c}
\|(\psi_1,V_1(t,\cdot) - V_2(t,\cdot)) + H_1 \ast (\psi_1, V_1(t,\cdot))(t) - H_2 \ast (\psi_1, V_2(t,\cdot))\|_{L^2(0, \tau)} \\ \\
\leq \|(\psi_1,V_1(t,\cdot) - V_2(t,\cdot)) + H_1 \ast (\psi_1, V_1(t,\cdot) - V_2(t,\cdot)) \|_{L^2(0, \tau)} \\ \\
+ \|(H_1 - H_2) \ast (\psi_1, V_2(t,\cdot))\|_{L^2(0, \tau)} \\ \\
\leq C_1 \tau^{1/2} \|V_1 - V_2\|_{H^{3/4}(0, \tau; L^2(\Omega))} \left(1 + Ê\tau^{1/2}\|H_1\|_{L^2(0, \tau)}\right))\\ \\
+ \|\psi_1\|_{L^2(\Omega)} \tau^{1/2}Ê\|H_1 - H_2\|_{L^2(0, \tau)}Ê\|V_2\|_{L^2(Q_\tau)}Ê\\ \\
\leq C_1 \tau^{1/2} \|V_1 - V_2\|_{H^{3/4}(0, \tau; L^2(\Omega))} \left(1 + Ê\tau^{1/2}\|H_1\|_{L^2(0, \tau)}\right) \\ \\
+ C_2 \tauÊ\|H_1 - H_2\|_{L^2(0, \tau)}Ê\|V_2\|_{H^{3/4}(0, \tau; L^2(\Omega))}Ê\\ \\
\leq C_3 \tau^{1/2} (1+ r + \|h^*\|_{L^2(0, T)} + \|V_0\|_{X_T}) (\|V_1 - V_2\|_{X_\tau} +  Ê\|H_1 - H_2\|_{L^2(0, \tau)}).
\end{array}
\end{equation}
Employing (\ref{eq3.16}), we obtain
\begin{equation}
\begin{array}{c}
\|- [(\psi_1,V_1(t,\cdot) - V_2(t,\cdot) ) + H_1 \ast (\psi_1, V_1(t,\cdot)) - H_2 \ast (\psi_1, V_2(t,\cdot))]Êz_0(x)\| _{H^{1/4,1/2}(Q_\tau)'}Ê\\ \\
\leq C_1 \tau^{1/4}Ê\|- [(\psi_1,V_1(t,\cdot) - V_2(t,\cdot) ) + H_1 \ast (\psi_1, V_1(t,\cdot)) - H_2 \ast (\psi_1, V_2(t,\cdot))]Êz_0(x)\| _{L^2(Q_\tau)}Ê\\ \\
\leq C_2 \tau^{3/4} \left(1+ r + \|h^*\|_{L^2(0, T)} + \|V_0\|_{X_T}\right) \left(\|V_1 - V_2\|_{X_\tau} +  Ê\|H_1 - H_2\|_{L^2(0, \tau)}\right).
\end{array}
\end{equation}
and
\begin{equation}
\begin{array}{c}
\|[(\psi_1,V_1(t,\cdot) - V_2(t,\cdot)) + H_1 \ast (\psi_1, V_1(t,\cdot)) - H_2 \ast (\psi_1, V_2(t,\cdot))] z_1(y)\|_{L^2(\Sigma_\tau)} \\ \\
 \leq C \tau^{1/2} \left(1+ r + \|h^*\|_{L^2(0, T)} + \|V_0\|_{X_T}\right) \left(\|V_1 - V_2\|_{X_\tau} +  Ê\|H_1 - H_2\|_{L^2(0, \tau)}\right).
\end{array}
\end{equation}
Next, we have
\begin{equation}
\begin{array}{c}
\| H_1 \ast AV_1 -  H_2 \ast AV_2\|_{H^{1/4,1/2}(Q_\tau)'} \\ \\
Ê\leq \| H_1 \ast A(V_1 - V_2)\|_{H^{1/4,1/2}(Q_\tau)'} + \| (H_1 - H_2) \ast AV_2\|_{H^{1/4,1/2}(Q_\tau)'} \\ \\
\leq C_1\left ( \|H_1\|_{L^1(0, \tau)}Ê\|A(V_1 - V_2)\|_{H^{1/4,1/2}(Q_\tau)'} + \|H_1-H_2\|_{L^1(0, \tau)}Ê\|A V_2\|_{H^{1/4,1/2}(Q_\tau)'} \right) \\ \\
\leq C_1 \tau^{1/2}Ê\big[\left(r + \|h^*\|_{L^2(0, T)}\right)Ê\|A(V_1 - V_2)\|_{H^{1/4,1/2}(Q_\tau)'} \\ \\
+ \|H_1-H_2\|_{L^2(0, \tau)}Ê\left(r + C_2\|AV_0\|_{H^{1/4,1/2}(Q_T)'}\right)\big],
\end{array}
\end{equation}
\begin{equation}
\begin{array}{c}
\|H_1 \ast BV_1  - H_2 \ast BV_2\|_{L^2(\Sigma_\tau)} \leq \|H_1 \ast B(V_1  - V_2)\|_{L^2(\Sigma_\tau)} + \|(H_1 - H_2) \ast B V_2\|_{L^2(\Sigma_\tau)} \\ \\
\leq \tau^{1/2}Ê\left(\|H_1\|_{L^2(0, \tau)} \|B(V_1  - V_2)\|_{L^2(\Sigma_\tau)} + \|H_1 - H_2\|_{L^2(0, \tau)}  \|B V_2\|_{L^2(\Sigma_\tau)}\right) \\ \\
\leq \tau^{1/2}Ê \left[\left(r + \|{h}^*\|_{L^2(0, T)}\right) \|B(V_1  - V_2)\|_{L^2(\Sigma_\tau)}
+ \left(r +  \|B V_0\|_{L^2(\Sigma_T)})\right \|H_1 - H_2\|_{L^2((0, \tau))}\right],
\end{array}
\end{equation}

\begin{equation}
\|(V_2 - V_1)_{|\Sigma_\tau}\|_{L^2(\Sigma_\tau)}Ê\leq \tau^{\frac{1}{2} - \frac{1}{p}} \|(V_2 - V_1)_{|\Sigma_\tau}\|_{L^p(0, \tau; L^2(\Gamma))},
\end{equation}
and, finally,
\begin{equation}\label{eq3.22}
\|\Psi(V_1) - \Psi(V_2)\|_{L^2(\Sigma_\tau)} \leq C(T) \tau^{1/2}Ê\|V_1 - V_2\|_{H^{3/4,3/2}(Q_\tau)}.
\end{equation}
From (\ref{eq3.16})-(\ref{eq3.22}), we deduce an estimate of the form
\begin{equation}
\|(v_1-v_2, h_1- h_2)\|_{Z_\tau} \leq C(r) \tau^\epsilon \|(V_1-V_2, H_1- H_2)\|_{Z_\tau},
\end{equation}
valid for every $(V_1,H_1)$ and $(V_2,H_2)$ in $Z_\tau(r)$, for some $\epsilon >0$. If we choose $\tau$ such that
$$
C(r) \tau^\epsilon < 1,
$$
and $\tau$ so small that $P$ carries $Z_\tau(r)$ into itself, we have that $P$ has a unique fixed point in $Z_\tau(r)$.

The proof is complete.
\end{proof}
To obtain global existence and uniqueness, we shall employ the following

\begin{lemma}\label{le3.2}
Assume that (K1)-(K6) are satisfied. Let $0 < \tau < \tau + \delta \leq \min\{T, 2\tau\}$ and let $(V,H) \in X_{\tau+\delta} \times L^2(0, \tau+\delta)$
be a solution of (\ref{eqA2.6}) (replacing $T$ with $\tau + \delta$). Setting
\begin{equation}\label{eq3.25B}
\begin{array}{cc}
w:= V_{|Q_{\tau}}, & v(t, \cdot):= V(\tau + t, \cdot),
\\ \\
h_0:= H_{|(0, \tau)}, & h(t):= H(\tau+t),
\end{array}
\end{equation}
there following propositions hold.

(I) $(v,h) \in X_\delta \times L^2(0, \delta)$ and solves the system
\begin{equation}\label{eq3.25}
\left\{\begin{array}{l}
D_tv(t,x)   = Av(t,x) + ({\widetilde h}_0 \ast {\widetilde A}w )(\tau + t,x) + (h \ast  Aw)(t,x) + (h_0 \ast Av)(t,x)  + v^*(\tau + t,x) \\ \\
 - [(\psi_1,v(t,\cdot)) + {\widetilde h}_0 \ast (\psi_1, \widetilde w)(\tau+t) + h \ast (\psi_1, w)(t) + h_0 \ast (\psi_1,v)(t)]Ê z_0(x), \quad (t,x) \in Q_\delta,  \\ \\
v(0,x) = w(\tau,x), \quad x \in \Omega \\ \\
Bv(t,y)  = - v(t,y) +\Psi(w,v)(t,y) - [({\widetilde h}_0 \ast {\widetilde B }w)(\tau+t,y) + (h \ast Bw) (t,y) + (h_0 \ast Bv)(t,y)]Ê\\ \\
+
 [(\psi_1,v(t,\cdot)) + \tilde h_0 \ast (\psi_1, \widetilde w)(\tau+t) + h \ast (\psi_1, w)(t) + h_0 \ast (\psi_1,v)(t)]Êz_1(y)Ê\\ \\
+ v^*_\Gamma (\tau+t,y), \quad (t,y) \in \Sigma_\delta \\ \\
h(t) = h^*(\tau+t) - [(\psi_1,v(t,\cdot)) + {\widetilde h}_0 \ast (\psi_1, \tilde w)(\tau+t) + h \ast (\psi_1, w)(t) + h_0 \ast (\psi_1,v)(t)], \\ \\
 t \in (0, \delta),
\end{array}
\right.
\end{equation}
where ${\widetilde h}_0$, ${\widetilde A}w$ and ${\widetilde B}w$ indicate the trivial extensions of $h_0$, $Aw$ and $Bw$ (see, in particular, (\ref{eq1.7A})).

(II) Let $(w,h_0) \in X_\tau \times L^2(0, \tau)$ be a solution of (\ref{eqA2.6}), with $\tau$ replacing $T$. Let $(v,h) \in X_\delta \times L^2(0, \delta)$
be a solution of (\ref{eq3.25}). Setting
$$
V(t,x):= \left\{\begin{array}{lll}
w(t,x), & \mbox{ if } & (t,x) \in Q_\tau, \\ \\
v(t-\tau,x), &\mbox{ if }  & (t,x) \in [\tau, \tau+\delta) \times \Omega,
\end{array}
\right.
$$
$$
H(t):= \left\{\begin{array}{lll}
h_0(t), & \mbox{ if } & t \in (0, \tau), \\ \\
h(t-\tau), & \mbox{ if } & t \in [\tau, \tau+\delta),
\end{array}
\right.
$$
then $(V,H) \in X_{\tau+\delta} \times L^2(0, \tau+\delta )$ and solves (\ref{eqA2.6}), where we have  replaced $T$ by $\tau+\delta$.
\end{lemma}

\begin{proof} (I) follows from Proposition \ref{pr1.2} and Lemma \ref{le1.8} (III).

Concerning (II), by Proposition \ref{pr1.2} we obtain $(V,H) \in X_{\tau+\delta} \times L^2(0, \tau+\delta)$ and solving
$$
D_t V = AV + f \quad {\mbox in }Ê\quad Q_{\tau + \delta},
$$
with
$$
\begin{array}{c}
f_{|(0, \tau)} = h_0 \ast Aw + v^*_{|(0, \tau)}Ê- [(\psi_1, w(t,\cdot)) + h_0 \ast (\psi_1, w(t,\cdot)]Êz_0(\cdot) \\ \\
= \{H \ast AV + v^* - [(\psi_1,V) + H \ast (\psi_1, V)]Êz_0\}_{|(0, \tau)},
\end{array}
$$
$$
\begin{array}{c}
f_{|(\tau, \tau+\delta)}(t,\cdot)Ê= ({\widetilde h}_0 \ast {\widetilde A}w ) (t,\cdot)
+ (h \ast  Aw)(t-\tau,\cdot) + (h_0 \ast Av)(t-\tau,\cdot)  + v^*_{|(\tau, \tau + \delta)}(t,\cdot) \\ \\
 - [(\psi_1,v(t-\tau,\cdot)) + {\widetilde h}_0 \ast (\psi_1, \widetilde w)(t) + h \ast (\psi_1, w)(t-\tau) + h_0 \ast (\psi_1,v)(t-\tau)]Ê z_0(\cdot) \\ \\
 = \{H \ast AV + v^* - [(\psi_1,V) + H \ast (\psi_1, V)]Êz_0\}_{|(\tau, \tau+\delta)},
 \end{array}
$$
as it is clear that $Aw = AV_{|(0, \tau)}$ and $Av = AV_{|(\tau, \tau + \delta)}(\cdot + \tau)$.
So, by Lemma \ref{le1.8}, we conclude that the first equation in (\ref{eqA2.6})
is satisfied if we replace $T$ by $\tau+\delta$. The validity of the other equations in (\ref{eqA2.6}) can be proved analogously.
\end{proof}

Now we are able to prove uniqueness.

\begin{lemma}\label{le3.3}
Assume (K1)-(K6). Then Problem (\ref{eqA2.6}) has, at most, one solution $(v,h) \in X_\tau \times L^2(0, \tau)$, $\forall \, \tau \in (0, T]$.
\end{lemma}
\begin{proof} Let $(V_1,H_1)$ and $(V_2,H_2)$ be solutions of (\ref{eqA2.6}) in $X_\tau \times L^2(0, \tau)$.
We observe that there exists $\tau_1 \in (0, \tau]$ such that $(V_1,H_1)$ and $(V_2,H_2)$ coincide in $Q_{\tau_1} \times (0, \tau_1)$.
This follows easily from Lemma \ref{le3.1}. In fact, by Lemma \ref{le1.3} (I), there exists $r >0$ such that, for $j \in \{1,2\}$ and
$\forall \,\sigma \in (0, \tau]$,
$$
\|(V_j - V_0)_{|Q_\sigma}\|_{X_\sigma} + \|(H_j - h^*)_{|(0, \sigma)}\|_{L^2(0, \sigma)} \leq r.
$$
Then, choosing $\tau_1 \leq \tau(r)$ (see Lemma \ref{le3.1}), we obtain that $(V_1,H_1)$ and $(V_2,H_2)$ coincide in $Q_{\tau_1} \times (0, \tau_1)$.
We choose $\tau_1$ as large as possible. More precisely, we set
\begin{equation}
\begin{array}{c}
\tau_1:= \sup \{\sigma \in (0, \tau]: \|V_1 - V_2\|_{X_\sigma} + \|H_1-  H_2\|_{L^2(0, \sigma)} = 0\} \\ \\
= \max \{\sigma \in (0, \tau]: \|V_1 - V_2\|_{X_\sigma} + \|H_1-  H_2\|_{L^2(0, \sigma)} = 0\}.
\end{array}
\end{equation}
We have to show that $\tau_1 = \tau$. We assume that $\tau_1 < \tau$. We shall see that there exists $\delta \in (0, \tau - \tau_1]$ such that
$(V_1,H_1)$ and $(V_2,H_2)$ coincide in $Q_{\tau_1 + \delta} \times (0, \tau_1 + \delta)$ and this contradicts the definition of $\tau_1$.
So that, consider
\begin{equation}
\delta \in (0, \min\{\tau_1, \tau - \tau_1\}].
\end{equation}
We introduce the new functions
\begin{equation}
\begin{array}{ccc}
w:= V_{1|Q_{\tau_1}} = V_{2|Q_{\tau_1}}, & v_1(t, \cdot):= V_1(\tau_1+t, \cdot), & v_2(t, \cdot):= V_2(\tau_1+t, \cdot),
\\ \\
h_0:= H_{1|(0, \tau_1)} = H_{2|(0, \tau_1)}, & h_1(t):= H_1(\tau_1+t), & h_2(t):= H_2(\tau_1+t).
\end{array}
\end{equation}
and we consider problem (\ref{eq3.25}), where we replace $v$ by $v_j$ and $h$ by $h_j$ ($j \in \{1,2\})$.  Setting
\begin{equation}
v:= v_1 - v_2, \quad h := h_1 - h_2,
\end{equation}
we obtain that $(v,h)$ satisfies the systen
\begin{equation}\label{eq3.30}
\left\{\begin{array}{l}
D_tv(t,x)   = Av(t,x) + (h \ast  Aw)(t,x) + (h_0 \ast Av)(t,x)   \\ \\
 - [(\psi_1,v(t,\cdot)) + h \ast (\psi_1, w)(t) + h_0 \ast (\psi_1,v)(t)]Ê z_0(x), \quad  (t,x) \in Q_\delta,  \\ \\
v(0,x) = 0, \quad x \in \Omega  \\ \\
Bv(t,y)  = - v(t,y) +\Psi(w,v_1)(t,y) -\Psi(w,v_2)(t,y)  - [ (h \ast Bw) (t,y) + (h_0 \ast Bv)(t,y)]Ê\\ \\
+ [(\psi_1,v(t,\cdot)) + h \ast (\psi_1, w)(t) + h_0 \ast (\psi_1,v)(t)]Êz_1(y), \quad (t,y) \in \Sigma_\delta \\ \\
h(t) = - [(\psi_1,v(t,\cdot))  + h \ast (\psi_1, w)(t) + h_0 \ast (\psi_1,v)(t)], \quad
 t \in (0, \delta).
\end{array}
\right.
\end{equation}
Following the same arguments as in the proofs of Lemma \ref{le3.1} and also of Corollary \ref{co2.1}, we deduce that
there exist  $C, \epsilon >0$, such that, if $\delta \in (0, \min\{\tau_1, \tau - \tau_1\}]$, then
\begin{equation}
\|v\|_{X_\delta} + \|h\|_{L^2(0, \delta)} \leq C \delta^\epsilon (\|v\|_{X_\delta} + \|h\|_{L^2(0, \delta)}).
\end{equation}
Choosing $\delta$ sufficiently small, this implies $\|v\|_{X_\delta} + \|h\|_{L^2(0, \delta)} = 0$.
\end{proof}
Now we want to show that (\ref{eqA2.6}) has a unique solution in $[0, T]$. To this aim, we consider the auxiliary system
\begin{equation}\label{eq3.32}
\left\{\begin{array}{l}
D_tv(t,x)   = Av(t,x)  + (h \ast  \zeta_0)(t,x) + (h_0 \ast Av)(t,x)  + z(t,x) \\ \\
 - [(\psi_1,v(t,\cdot))  + h \ast \chi_1(t) + h_0 \ast (\psi_1,v)(t)]Ê z_0(x), \quad  (t,x) \in Q_T,  \\ \\
v(0,x) = w(\tau,x), \quad x \in \Omega,  \\ \\
Bv(t,y)  = - v(t,y) + \Psi(w,v)(t,y) - [(h \ast g_0) (t,y) + (h_0 \ast Bv)(t,y)]Ê\\ \\
+  [(\psi_1,v(t,\cdot))  + h \ast \chi_1(t) + h_0 \ast (\psi_1,v)(t)]Êz_1(y)Ê
+ g(t,y), \quad (t,y) \in \Sigma_T \\ \\
h(t) = k(t)  - [(\psi_1,v(t,\cdot))  + h \ast \chi_1(t) + h_0 \ast (\psi_1,v)(t)]Ê, \quad t \in (0, T).
\end{array}
\right.
\end{equation}

\begin{lemma}\label{le3.4} Assume that (H1)-(H2)  hold. Consider system (\ref{eq3.32}), where $\zeta_0,\, z \in H^{1/4,1/2}(Q_T)'$,
$h_0, \,k, \,\chi_1 \in L^2(0, T)$, $\psi_1, \,z_0 \in L^2(\Omega)$, $g_0,\, g \in L^2(\Sigma_T)$, $z_1 \in L^2(\Gamma)$, $w \in X_\tau$, for some $\tau >0$.
Then (\ref{eq3.32}) admits a unique solution $(v,h)$ in $X_T \times L^2(0, T)$.
\end{lemma}

\begin{proof} We prove the result in two steps. First, we show that there exists $\delta \in (0, T]$, independent of $z$, $w$, $g$ and $k$, such that
(\ref{eq3.32}) has a unique solution in $X_\delta \times L^2(0, \delta)$. This can be proved as follows. Set
$$
Z_\delta := \{(V,H) \in X_\delta \times L^2(0, \delta) : V(0,\cdot) = w(\tau,\cdot)\}, \quad  0 < \delta \leq T,
$$
which is a closed subset of $X_\delta \times L^2(0, \delta)$. If $(V, H) \in Z_\delta$,  we consider the solution $(v,h) \in Z_\delta$ of
$$
\left\{\begin{array}{l}
D_tv(t,x)   = Av(t,x)  + (H \ast  \zeta_0)(t,x) + (h_0 \ast AV)(t,x)  + z(t,x) \\ \\
 - [(\psi_1,V(t,\cdot))  + H \ast \chi_1(t) + h_0 \ast (\psi_1,V)(t)]Ê z_0(x), \quad  (t,x) \in Q_T,  \\ \\
v(0,x) = w(\tau,x), \quad x \in \Omega,  \\ \\
Bv(t,y)  = - V(t,y) + \Phi(w,V)(t,y) - [(H \ast g_0) (t,y) + (h_0 \ast BV)(t,y)]Ê\\ \\
+
  [(\psi_1,V(t,\cdot))  + H \ast \chi_1(t) + h_0 \ast (\psi_1,V)(t)]Êz_1(y)Ê
+ g(t,y), \quad (t,y) \in \Sigma_T \\ \\
h(t) = k(t)  - [(\psi_1,V(t,\cdot))  + H \ast \chi_1(t) + h_0 \ast (\psi_1,V)(t)]Ê, \quad t \in (0, \delta).
\end{array}
\right.
$$
We can show that, if $\delta$ is sufficiently small, the mapping $(V,H) \to (v,h)$ admits a unique fixed point.
In fact, considering $(V_j,H_j) \to (v_j,h_j)$ ($j \in \{1,2\}$), then we have
$$
\left\{\begin{array}{l}
D_t(v_1-v_2)(t,x)   = A(v_1 - v_2) (t,x)  + [(H_1 - H_2) \ast  \zeta_0)(t,x) + [h_0 \ast A(V_1 - V_2)](t,x)  \\ \\
 - [(\psi_1,V_1(t,\cdot) - V_2(t,\cdot))  + (H_1 - H_2) \ast \chi_1(t) + h_0 \ast (\psi_1,V_1 - V_2)(t)]Ê z_0(x), \quad  (t,x) \in Q_T,  \\ \\
v_1(0,x) - v_2(0,x) = 0,  \quad x \in \Omega,  \\ \\
B(v_1 - v_2)(t,y)  = - V_1(t,y) + V_2(t,y) + \Phi(w,V_1)(t,y) - \Phi(w,V_2)(t,y) \\ \\
 - [((H_1 - H_2) \ast g_0) (t,y) + (h_0 \ast B(V_1 - V_2))(t,y)]Ê\\ \\
+  [(\psi_1,V_1(t,\cdot) - V_2(t,\cdot))  + (H_1 - H_2) \ast \chi_1(t) + h_0 \ast (\psi_1,V_1 - V_2)(t)]Êz_1(y), \quad (t,y) \in \Sigma_T, \\ \\
h_1(t) - h_2(t) =   - [(\psi_1,V_1(t,\cdot) - V_2(t, \cdot))  + (H_1 - H_2) \ast \chi_1(t) + h_0 \ast (\psi_1,V_1 - V_2)(t)]Ê, \quad t \in (0, \delta),
\end{array}
\right.
$$
so that
$$
\|v_1 - v_2\|_{X_\delta} + \|h_1 - h_2\|_{L^2(0, \delta)} \leq C \delta^\epsilon \left(\|V_1 - V_2\|_{X_\delta} + \|H_1 - H_2\|_{L^2(0, \delta)}\right),
$$
for some $\epsilon >0$, and $C$ and $\epsilon$ independent of $z$, $w$, $g$, $k$ (cf. Corollary \ref{co2.1}).
Hence, if $\delta$ is sufficiently small, problem (\ref{eq3.32}) has a unique solution $(v,h) \in X_\delta \times L^2(0, \delta)$.
Observe that, in case $\delta < T$, we can extend $(v, h)$ to $(0,T)$ on account of Proposition \ref{pr1.2} and Lemma \ref{le1.8} (II).
Indeed, taking as new unknowns $v_1(t, \cdot):= v(\delta + t, \cdot)$ and $h_1(t):= h(\delta + t)$, then $(v_1,  h_1)$ satisfies
\begin{equation}\label{eq3.33}
\left\{\begin{array}{l}
D_t  v_1(t,x)   = A  v_1(t,x)  + ( h_1 \ast \zeta_0)(t,x) + ({\widetilde h}_{|(0, \delta)} \ast  \zeta_0)(t +\delta,x)  \\ \\
+ (h_0 \ast Av_1)(t,x) +  (h_0 \ast {\widetilde A} v_{|(0, \delta)})(t + \delta,x) + z(t + \delta,x) \\ \\
 - [(\psi_1,v_1(t,\cdot))  + h_1 \ast \chi_1(t) + ({\widetilde h}_{|(0, \delta)} \ast \chi_1) (t+\delta) +  h_0 \ast (\psi_1,v_1)(t) \\ \\
 + h_0 \ast (\psi_1,{\widetilde v}_{|(0, \delta)})(t+\delta)]Ê z_0(x), \quad  (t,x) \in Q_{\min\{\delta, T-\delta\}},  \\ \\
v_1(0,x) = w_1(\tau+\delta,x) = v(\delta, x), \quad x \in \Omega, \\ \\
Bv_1(t,y)  = - v_1(t,y) + \Phi(w_1,v_1)(t,y) - [(h_1 \ast g_0) (t,y) + ({\widetilde h}_{|(0, \delta)} \ast g_0) (t+\delta,y) \\ \\
+ (h_0 \ast Bv_1)(t,y) + (h_0 \ast {\widetilde B}v_{|(0, \delta)})(t+\delta,y)] \\ \\
+    [(\psi_1,v_1(t,\cdot))  + h_1 \ast \chi_1(t) + ({\widetilde h}_{|(0, \delta)} \ast \chi_1) (t+\delta) +  h_0 \ast (\psi_1,v_1)(t) \\ \\
 + h_0 \ast (\psi_1,{\widetilde v}_{|(0, \delta)})(t+\delta)]
 Êz_1(y)Ê + g(t+\delta,y), \quad (t,y) \in \Sigma_{\min\{\delta, T-\delta\}}, \\ \\
h_1(t) = k(t+\delta)  -  [(\psi_1,v_1(t,\cdot))  + h_1 \ast \chi_1(t) + {\widetilde h}_{|(0, \delta)} \ast \chi_1 (t+\delta) +  h_0 \ast (\psi_1,v_1)(t) \\ \\
 + h_0 \ast (\psi_1,{\widetilde v}_{|(0, \delta)})(t+\delta)], \quad t \in (0, \min\{\delta, T-\delta\}),
\end{array}
\right.
\end{equation}
with
$$
w_1(t,x) = \left\{\begin{array}{ll}
w(t,x) & \mbox{ if }Êt \in [0, \tau], \\ \\
v(t-\tau,x) & \mbox{ if }Êt \in [\tau, \tau +\delta].
\end{array}
\right.
$$
Now we observe that (\ref{eq3.33}) is a system of the same form of (\ref{eq3.32}). In fact it suffices to replace $v$ by $v_1$, $h$ by $h_1$, $z(t,x)$ by
$$({\widetilde h}_{|(0, \delta)} \ast \zeta_0)(t +\delta,x) +
  (h_0 \ast {\widetilde A} v_{|(0, \delta)})(t + \delta,x) + z(t + \delta,x) - [ {\widetilde h}_{|(0, \delta)} \ast \chi_1 (t+\delta)
 + h_0 \ast (\psi_1,{\widetilde v}_{|(0, \delta)})(t+\delta)]Ê z_0(x),
$$
$w$ by $w_1$, $g(t,y)$ by
$$
- [ {\widetilde h}_{|(0, \delta)} \ast g_0 (t+\delta,y) + (h_0 \ast {\widetilde B}v_{|(0, \delta)})(t+\delta,y)]
+ [{\widetilde h}_{|(0, \delta)} \ast \chi_1 (t+\delta) + h_0 \ast (\psi_1,{\widetilde v}_{|(0, \delta)})(t+\delta)]
 Êz_1(y)Ê+ g(t+\delta,y),
$$
and $k(t)$ by
$$k(t+\delta) - [{\widetilde h}_{|(0, \delta)} \ast \chi_1 (t+\delta) + h_0 \ast (\psi_1, {\widetilde v}_{|(0, \delta)})(t+\delta)].$$
So, following the same arguments as in the first part of the proof, we can determine a unique solution $(v_1,h_1)$ in
$X_{\min\{\delta,T-\delta\}} \times L^2(0, \min\{\delta,T-\delta\})$. Now, setting
$$
v(t,x):= v_1(t-\delta,x), \quad h(t) := h_1(t-\delta), \quad \mbox{Êfor }Êt \in (\delta, \min\{2\delta, T\}),
$$
and applying Proposition \ref{pr1.2}, together with Lemma \ref{le1.8}, we obtain a unique solution in
$X_{\min\{2\delta, T\}} \times L^2(0, \min\{2\delta, T\})$ (see also the proof of Lemma \ref{le3.2} (II)).
In case we have $2\delta < T$, we can iterate the method and, in a finite number of steps, we can construct a solution in $[0, T]$.
\end{proof}

To conclude, we state the main result of this section
\begin{theorem}\label{th6.6}
Assume (K1)-(K6). Then (\ref{eqA2.6}) has a unique solution in $X_T \times L^2(0, T)$.
\end{theorem}
\begin{proof} The uniqueness was already proved in Lemma \ref{le3.3}.
Concerning the existence, we have already seen in Lemma \ref{le3.1} that there exists a solution $(w,h_0) \in X_\tau \times L^2(0, \tau)$,
for some $\tau \in (0, T]$. If  $\tau < T$, we employ Lemma \ref{le3.2} in order to extend $(w,h_0)$ from $(0, \tau)$ to $(0, \min\{2\tau, T\})$.
To this aim, we consider system (\ref{eq3.25}), which is of the form (\ref{eq3.32}) whence we set
$\zeta_0:= Aw$, $\chi_1(t):= - (\psi_1, w(t,\cdot))$, $g_0(t,y) := Bw(t,y)$,
$$
z(t,x):= ({\widetilde h}_0 \ast {\widetilde A}w )(T + t,x) + v^*(\tau + t,x) - [{\widetilde h}_0 \ast (\psi_1, {\widetilde w})(\tau + t) ]z_0(x),
$$
$$
g(t,y):= -({\widetilde h}_0 \ast {\widetilde B}w)(\tau + t,y) + [{\widetilde h}_0 \ast (\psi_1, {\widetilde w})(\tau + t)]Êz_1(y)Ê+ v^*_\Gamma (\tau + t,y),
$$
$$
k(t):=  h^*(\tau + t)  -  {\widetilde h}_0 \ast (\psi_1, {\widetilde w})(\tau + t).
$$
So, by Lemma \ref{le3.4}, we can extend $(w, h_0)$ to a solution of (\ref{eqA2.6}) in $X_{\min\{2\tau,T\}} \times L^2(0, \min\{2\tau,T\})$.
If  $2\tau < T$, we iterate the procedure, replacing $\tau$ by $2\tau$, and in a finite number of steps we get the result.
\end{proof}

\section{Proof of the main result}\label{se7}

\setcounter{equation}{0}

Now we are able to prove the main result of the paper, namely Theorem \ref{thA2.1}.
By Proposition \ref{prA2.1}, we are reduced to search for a solution $(v,h) \in H^{1,2}(Q_T) \times L^2(0, T)$ to Problem \ref{P3}.
In Section \ref{se6}, we have just seen that there exists a unique solution $(v,h) \in X_T \times L^2(0, T)$.
So, it remains to show that, if (H1)-(H9) are satisfied, then $v$ belongs, in fact, to  $H^{1,2}(Q_T)$.
To this aim, we begin with the following

\begin{lemma} \label{le3.5} Taking $h \in L^2(0, T)$, $f \in H^{1/4,1/2}(Q_T)'$, $v_0 \in H^{1/2}(\Omega)$, $g \in L^2(\Sigma_T)$, consider the system
\begin{equation}\label{eq3.34}
\left\{\begin{array}{l}
D_tv(t,x)   = Av(t,x) + h \ast Av(t,x) + f(t,x), \quad  (t,x) \in Q_\tau,  \\ \\
v(0,x) = v_0(x), \quad x \in \Omega,  \\ \\
Bv(t,y)  = - h \ast Bv(t,y) Ê
+ g(t,y), \quad (t,y) \in \Sigma_\tau,
\end{array}
\right.
\end{equation}
Then the following propositions hold.

(I) (\ref{eq3.34}) has a unique solution $v$ in $X_T$.

(II) If $f \in L^2(Q_T)$, $v_0 \in H^{1}(\Omega)$, and $g \in H^{1/4,1/2}(\Sigma_T)$, then $v$ belongs to $H^{1,2}(Q_T)$.
\end{lemma}

\begin{proof} The proof of (I) is a simplified variation of the proof of Lemma \ref{le3.4} and we leave it to the reader.

Concerning (II), we follow again the idea contained in the proof of Lemma \ref{le3.4}. Let $\delta \in (0, T]$ and set
\begin{equation}
Z_\delta := \{V \in H^{1,2}(Q_\delta) : V(0,\cdot) = v_0\},
\end{equation}
which is a closed subset of $H^{1,2}(Q_\delta)$.  If $V \in Z_\delta$, we consider the solution $v \in Z_\delta$ of
$$
\left\{\begin{array}{l}
D_tv(t,x)   = Av(t,x) + h \ast AV(t,x) + f(t,x), \quad  (t,x) \in Q_\delta,  \\ \\
v(0,x) = v_0(x),  \quad x \in  \Omega\\ \\
Bv(t,y)  = - h \ast BV(t,y) Ê
+ g (t,y), \quad (t,y) \in \Sigma_\delta.
\end{array}
\right.
$$
We observe that $h \ast AV \in L^2(Q_\delta)$ and, on account of Theorem \ref{th1.1} (II) and Lemma \ref{le1.6}, we have also
$h \ast BV \in H^{1/4,1/2}(\Sigma_\delta)$. We can show that, if $\delta$ is sufficiently small, the mapping $V \to v$ has a unique fixed point.
In fact, considering $V_j \to v_j$ ($j \in \{1,2\}$), then we have
$$
\left\{\begin{array}{l}
D_t(v_1 - v_2)(t,x)   = A(v_1 - v_2)(t,x) + h \ast A(V_1 - V_2)(t,x), \quad  (t,x) \in Q_\delta,  \\ \\
(v_1 - v_2)(0,x) = 0, \quad x \in \Omega  \\ \\
B(v_1 - v_2)(t,y)  = - h \ast B(V_1 - V_2)(t,y), \quad (t,y) \in \Sigma_\delta.
\end{array}
\right.
$$
So that, if we indicate by $C_1, C_2, ...$ some positive constants which are independent of $\delta$, $V_1$, $V_2$, by an application
of Proposition \ref{pr1.4} and Lemma \ref{le1.6} we deduce
$$
\begin{array}{c}
\|v_1 - v_2\|_{H^{1,2}(Q_\delta)} \leq C_1(\|h \ast A(V_1 - V_2)\|_{L^2(Q_\delta)} + \| h \ast B(V_1 - V_2)\|_{H^{1/4,1/2}(\Sigma_\delta)}) \\ \\
\leq C_2\|h\|_{L^1(0, \delta)} (\|V_1 - V_2\|_{L^2(0, \delta; H^2(\Omega))} +  \|B(V_1 - V_2)\|_{H^{1/4,1/2}(\Sigma_\delta)}) \\ \\
\leq C_3 \delta^{1/2}Ê\|V_1 - V_2\|_{H^{1,2}(Q_\delta)}.
\end{array}
$$
Choosing $\delta >0$ such that $C_3 \delta^{1/2}Ê < 1$, then problem (\ref{eq3.34}) admits a unique solution $v$ in $H^{1,2}(Q_\delta)$.
We observe that $\delta$ is independent of $f$, $v_0$, $g$. In case $\delta < T$, in order to extend $v$ we employ again Lemma \ref{le1.8}.
Taking as new unknown  $v_1(t,\cdot):= v(\delta + t, \cdot)$, then $v_1$ should satisfy
$$
\left\{\begin{array}{l}
D_tv_1(t,x)   = Av_1(t,x) + h \ast Av_1(t,x) + [h \ast {\widetilde A}v_{|Q_\delta}] (t+T,x) + f(t+\delta,x), \quad  (t,x) \in Q_\tau,  \\ \\
v_1(0,x) = v(\delta,x),  \quad x \in \Omega\\ \\
Bv_1(t,y)  = - h \ast Bv_1(t,y) Ê- [h \ast {\widetilde B}v_{|Q_\delta}] (t+T,y) + g(t+\delta,y), \quad (t,y) \in \Sigma_\tau,
\end{array}
\right.
$$
which has the same structure of (\ref{eq3.34}), when we replace $f(t,x)$ by $[h \ast {\widetilde A}v_{|Q_\delta}] (t+T,x) + f(t+\delta,x)$,
$v_0$ by $v(\delta,\cdot)$ and $g(t,y)$ by Ê$- [h \ast {\widetilde B}v_{|Q_\delta}] (t+T,y) + g(t+\delta,y)$.
Hence, we can extend $v$ to a solution of domain $Q_{\min\{2\delta,T\}}$. If $2\delta < T$, we iterate the procedure and, in a finite number
of steps, we get the proof.
\end{proof}

In order to conclude, we need some auxiliary results.

\begin{lemma}\label{le2.2}
$BV([0, T]) \hookrightarrow H^\beta(0, T)$, $\forall \, \beta \in [0, 1/2)$.
\end{lemma}

\begin{proof} By Theorem 10 in \cite{Si1}, we have $W^{1,1}(0, T) \hookrightarrow H^\beta(0, T)$, $\forall \, \beta \in [0, 1/2)$.
Moreover, it is well known that $W^{1,1}(0, T) \hookrightarrow BV([0, T])$ and that  $V_0^T(f) = \int_0^T |f'(t)| dt$, $\forall \, f \in W^{1,1}(0, T)$. Now, let us take $f \in BV([0, T])$.
Extending $f$ to $F:\R \to \C$ by $F(t) = f(0)$ if $t < 0$ and $F(t) = f(T)$ if $t > T$, we obtain
 $F \in BV(\R)$, with variation
$$V(F) = |f(0)| + V_0^T (f) + |f(T)|.$$
Now we fix $\omega \in \mathcal D (\R)$ with $\int_\R \omega (t) dt = 1$, and set, for $k \in \N$, $t \in \R$, $\omega_k(t) := k \omega(kt)$. $\{\omega_k\}_{k \in \N}$ is a sequence of standard  smooth
mollifiers converging to $\delta$ in the sense of distributions. Taking $f_k = (F \ast \omega_k)_{|[0, T]}$, we have $f_k \in W^{1,1}(0, T)$ , $\|f_k - f\|_{L^1((0, T))} \to 0$ ($k \to \infty$), so that,
(possibly passing to a subsequence)
$$
f_k(t) \to f(t) \quad (k \to \infty) \quad \text{almost everywhere in } [0,T].
$$
If $0 = t_0 < t_1 < \dots ... < t_{N-1} < t_N = T$, we have
$$
\begin{array}{c}
\sum_{j=1}^N |f_k(t_j) - f_k(t_{j-1})| \leq \int_\R \sum_{j=1}^N |F(t_j - s) - F(t_{j-1} - s)| |\omega_k(s)| ds \\ \\
\leq V(F) \|\omega_k\|_{L^1(\R)} =  V(F) \|\omega\|_{L^1(\R)},
\end{array}
$$
implying
$$
\|f_k'\|_{L^1((0, T)} = V_0^T(f_k) \leq V(F) \|\omega\|_{L^1(\R)} \quad \forall k \in \N.
$$

So, if $\beta \in [0, 1/2)$, by the lemma of Fatou, we have
$$
\begin{array}{c}
[f]_{H^\beta(0,T)} = \int_0^T \int_0^t \frac{|f(t) - f(s)|^2}{(t-s)^{1+2\beta}} ds dt
\leq \liminf_{k \to \infty}Ê\int_0^T \int_0^t \frac{|f_k(t) - f_k(s)|^2}{(t-s)^{1+2\beta}} ds dt  \\ \\
= \liminf_{k \to \infty} [f_k]_{H^\beta(0,T)} \leq C(\beta)  \liminf_{k \to \infty} \|f_k\|_{W^{1,1}((0, T))}Ê\\ \\
\leq C(\beta) (\|f\|_{L^1((0, T))} + V(F) \|\omega\|_{L^1(\R)}).
\end{array}
$$
The conclusion follows.
\end{proof}

\begin{lemma}\label{le2.3}
Let $V$ be a Hilbert space and $0 < \beta < \alpha < 1$. Assume $\phi \in H^\beta(0, T)$ and $f \in C^\alpha([0, T]; V)$,
or $\phi \in C^\alpha([0, T])$ and $f \in H^\beta(0, T; V)$. Then $\phi f \in  H^\beta(0, T; V)$.
\end{lemma}

\begin{proof} We prove only the first case, the other case can be treated analogously. Indeed, there holds
$$
\begin{array}{c}
|f|_{H^\beta(0,T;V)} = \int_0^T (\int_0^t \frac{\|\phi(t) f(t) - \phi(s) f(s)\|_V^2}{(t-s)^{1+2\beta}} \\ \\
\leq 2\left [\int_0^T |\phi(t)|^2 (\int_0^t \frac{\| f(t) -  f(s)\|_V^2}{(t-s)^{1+2\beta}} ds) dt
+ \int_0^T (\int_0^t \frac{| \phi (t) -  \phi(s)|^2}{(t-s)^{1+2\beta}} ds) \|f(t)\|_V^2 dt \right] \\ \\
\leq 2\Big[\int_0^T |\phi(t)|^2 ( [f]_{C^\alpha([0, T]; V)}^2 \int_0^t (t-s)^{2(\alpha - \beta) - 1}Êds ) dt
\\ \\
+ \|f\|_{L^\infty(0, T; V)}^2 \int_0^T (\int_0^t \frac{| \phi (t) -  \phi(s)|^2}{(t-s)^{1+2\beta}} ds) dt \Big] \\ \\
\leq  2\left[ \frac{T^{2(\alpha - \beta)}}{2(\alpha - \beta)}Ê[f]_{C^\alpha([0, T]; V)}^2 \int_0^T |\phi(t)|^2  dt
+ \|f\|_{L^\infty(0, T; V)}^2 \int_0^T (\int_0^t \frac{| \phi (t) -  \phi(s)|^2}{(t-s)^{1+2\beta}} ds) dt \right].
\end{array}
$$
\end{proof}

\begin{corollary}\label{co2.2}
Assume (C1)-(C3) and $\tau  \in (0, T]$. For any $v \in H^{3/4,3/2}(Q_\tau)$, then $\Psi(v) \in H^{1/4,1/2}(\Sigma_\tau)$.
\end{corollary}

\begin{proof} We recall that $\Psi(v)(t,y) = D_t[(E_1 * r_v)(t)u_A(t,y) + E_0(t,y)]\,$ where
$r_v := \mW (\mM (u_0 + 1 \ast v)) \in C([0, \tau]) \cap BV([0, \tau])$ and
$E_0(t,y) = [ (E_1 \ast u_C)(t) + \epsilon \phi_0 E_1(t)]Êu_A(t,y) + u_B(t,y)$.
On account of (H8) and (H9), there hold  $D_t (E_1 \ast r_v), \, [(E_1 \ast u_C)(t) + \epsilon \phi_0 E_1(t)] \in C([0, \tau]) \cap BV([0, \tau])$ and, moreover,
$$u_A \in H^{5/4}(0, T; L^2(\Gamma)) \hookrightarrow H^{1}(0, T; L^2(\Gamma))  \hookrightarrow C^{1/2}([0, T]; L^2(\Gamma)),
\quad u_B\in H^{5/4}(0, T; L^2(\Gamma)).$$
So, an applications of  Lemmas \ref{le2.2}Ê and \ref{le2.3} gives $\Psi(v) \in H^{1/4}(0, T; L^2(\Gamma))$. Moreover, using again (H8) and (H9), it is
easy to realize that $\Psi(v) \in L^2(0, T;H^{1/2}(\Gamma))$.
\end{proof}

Now we are able to prove the following final result, which completes the proof of Theorem \ref{thA2.1}.

\begin{theorem}
Assume (K1)-(K6), and, moreover, $v^* \in L^2(Q_T)$, $v_0 \in H^1(\Omega)$, $z_1 \in H^{1/2}(\partial \Omega)$, $v^*_\Gamma \in H^{1/4,1/2}(\Sigma_T)$.
Then the solution $(v,h) \in X_T \times L^2(0, T)$ of (\ref{eqA2.6})  belongs, in fact, to $H^{1,2}(Q_T) \times L^2(0, T)$.
\end{theorem}

\begin{proof} Indeed, $v$ is the solution in $X_T$ of (\ref{eq3.34}), if we set
$$
f(t,x):=v^*(t,x)
 - [(\psi_1,v(t,\cdot)) + h \ast (\psi_1, v(t,\cdot))]Êz_0(x),
$$
and
$$
g(t,y):= - v_{|\Sigma_T}(t,y) + \Psi(v)(t,y) +    [(\psi_1,v(t,\cdot)) + h \ast (\psi_1, v(t,\cdot))] z_1(y)Ê\\ \\
+ v^*_\Gamma (t,y).
$$
Clearly, as $v^* \in L^2(Q_T)$, then $f \in L^2(Q_T)$. Moreover, as $v \in H^{3/4, 3/2}(Q_T)$ then $v_{|\Sigma_T} \in H^{1/2,1}(\Sigma_T)$.
By Corollary \ref{co2.2}, we have $\Psi(v) \in H^{1/4,1/2}(\Sigma_T)$.
Thanks to Lemma \ref{le1.6}, it holds $ (\psi_1,v(t,\cdot)) + h \ast (\psi_1, v(t,\cdot)) \in H^{1/4}(0, T)$, so
that $ [(\psi_1,v(t,\cdot)) + h \ast (\psi_1, v(t,\cdot))] z_1(y) \in H^{1/4,1/2}(\Sigma_T)$.
Hence, as $v^*_\Gamma \in H^{1/4,1/2}(\Sigma_T)$, we deduce $g \in H^{1/4,1/2}(\Sigma_T)$.
And finally from Lemma \ref{le3.5} (II) it follows that $v \in H^{1,2}(Q_T)$.
\end{proof}

\begin{remark}\label{ref}
{\rm Examples of memory operators $\mathcal W$ satisfying (C1) and (C3) are Preisach operators and generalized play operators
(with fixed initial value of the output), under suitable conditions.

Concerning Preisach operators, conditions implying that, $\forall \, \tau \in [0, T]$, $\mathcal W_\tau: C([0, \tau]) \to C([0, \tau])$
in a uniformly Lipschitz way are given in \cite{Vi1}, Chap. IV, Thms 3.1 and 3.5. Analogous
results for generalized play operators can be found again in \cite{Vi1}, Chap. III, Thm. 2.2.

By inspection of the previous proofs, what we need is that, $\forall \,\tau \in [0, T]$, $\mathcal W_\tau:
C([0, \tau]) \to C([0, \tau])$ in a uniformly Lipschitz way and
\begin{equation}\label{eq7.3}
\mathcal W_\tau(H^{3/2}(0, \tau)) \subset H^{1/4}(0, \tau) \cap C([0, \tau]),
\end{equation}
which is guaranteed by (C1)-(C3). In suitable circumstances, the mentioned operators map $W^{1,1}(0, \tau)$ into itself boundedly
(see \cite{Vi1}, Chap. IV, Thm. 3.10 and Chap. III, Thm. 2.3). Now, we can easily see that, if $\mathcal W_\tau$ maps $C([0, \tau])$
into itself continuously and $W^{1,1}(0, \tau)$ into itself boundedly, then $\mathcal W_\tau$ maps $C([0, \tau]) \cap BV([0, \tau])$ into itself.
In fact, if $r \in C([0, \tau]) \cap BV([0, \tau])$, then we can construct a bounded sequence $(r_k)_{k \in \N}$  in $W^{1,1}(0, \tau)$,
converging to $r$ uniformly in $[0, \tau]$. From the lower semicontinuity of $V_0^\tau(\cdot)$ in $C([0, \tau])$, we deduce
$$
V_0^\tau(\mathcal W_\tau (r)) \leq \liminf_{k \to \infty}ÊV_0^\tau(\mathcal W_\tau (r_k))  < \infty.
$$
Moreover, we can also observe that condition $\mathcal W_\tau (W^{1,1}(0, \tau)) \subset W^{1,1}(0, \tau)$ implies (\ref{eq7.3}).

Finally,  as $H^{3/2}(0, \tau)) \subset C^\alpha([0, \tau])$ $\forall \,\alpha \in [0, 1)$, and $C^{1/4}([0, \tau])
\subset H^{1/4}(0, \tau) \cap C([0, \tau])$, then we can obtain (\ref{eq7.3}) requiring that $\mathcal W_\tau
(C^\alpha([0, \tau]) \subset C^{1/4}([0, \tau])$, for some $\alpha \in [0, 1)$.
Concerning the Preisach operator, this can be obtained imposing the assumptions indicated in \cite{Vi1}, Chap. IV, Thm. 3.9.
}

\end{remark}


\begin{thebibliography}{99}

\bibitem{Ad1} R. A. Adams, {\it Sobolev spaces}, Academic Press (1975).

\bibitem{Am1} H. Amann, {\it Anisotropic Functions Spaces and Maximal Regularity for Parabolic Equations. Part 1: Function Spaces},
Institut f\"ur Mathematik, Universit\"at Z\"urich.

\bibitem{BeLo1} J. Bergh, J. L\"ofstrom, {\it Interpolation Spaces. An Introduction}, Springer-Verlag (1976).

\bibitem{CaCo1} C. Cavaterra, F. Colombo, ``Identifying a heat source in automatic control problems",  Commun. Appl. Nonlinear Anal. {\bf 11} (2004), 1-23.

\bibitem{CoGrSp1} P. Colli, M. Grasselli, J. Sprekels, ``Automatic Control via Thermostats of a Hyperbolic Stefan Problem with Memory",
Appl. Math. Optim. {\bf 39} (1999), 229-255.

\bibitem{CoGu1} F. Colombo, D. Guidetti, ``A global in time existence and uniqueness result for a semilinear integrodifferential inverse problem in Sobolev spaces",
Math. Models Methods Appl. Sci. {\bf 17} (2007), 537-565.

\bibitem{Gr1} P. Grisvard, ``Spazi di tracce e applicazioni", Rend. Mat. ({\bf 5}) (1972), 657-729.

\bibitem{HoNiSp1} K.H. Hoffmann, M. Niezg\'odka, J. Sprekels, ``Feedback control via thermostats of multidimensional two-phase Stefan problems, Nonlinear Anal.
{\bf 15} (1990), 955-976.

\bibitem{KrPo1} M. Krasnosel'skii, A. Pokrovskii, {\it Systems with Hysteresis}, Springer-Verlag (1980).

\bibitem{LiMa1} J. L. Lions, E. Magenes, {\it  Probl$\grave{e}$mes aux limites non homog$\grave{e}$nes et applications}, vol. I, Dunod (1968).

\bibitem{LiMa2} J. L. Lions, E. Magenes, {\it  Nonhomogeneous boundary value problems and applications}, vol. II, Springer-Verlag (1972).

\bibitem{LoSi1} A. Lorenzi, E. Sinestrari, ``An inverse problem in the theory of materials with memory", Nonlinear Anal. {\bf 12} (1988), 1317-1335.

\bibitem{Si1}ÊJ. Simon, ``Sobolev, Besov and Nikolskii Spaces: Imbeddings and Comparisons for Vector Valued Spaces on an Interval",
Ann. Mat. Pura Appl. {\bf 157} (1990), 117-148.

\bibitem{Vi1} A. Visintin, {\it Differential Models for Hysteresis}, Applied Mathematical Sciences vol. 111, Springer-Verlag (1994).

\end{thebibliography}
\end{document}